\documentclass[12pt,twoside,reqno]{amsart}
\usepackage{amssymb}
\usepackage{amsfonts}
\usepackage{amsmath}
\usepackage{amsthm}
\usepackage[dvips,bottom=1.4in,right=1in,top=1in, left=1in]{geometry}

\newtheorem{theorem}{Theorem}[section]

\newtheorem{corollary}[theorem]{Corollary}

\newtheorem{definition}[theorem]{Definition}
\newtheorem{example}[theorem]{Example}

\newtheorem{lemma}[theorem]{Lemma}

\newtheorem{proposition}[theorem]{Proposition}
\newtheorem{remark}[theorem]{Remark}

\numberwithin{equation}{section}
\input{tcilatex}
\subjclass[2000]{35J20, 35J25, 35J60, 35J65, 49J27, 52A41.}
\keywords{Laplace-Beltrami operator, global constraints, nonlinear
elliptic boundary value problems at resonance, nonlinear boundary
conditions, Fredholm alternative.}

\begin{document}
\title[Fredholm alternative]{Fredholm alternative, semilinear elliptic
problems, and Wentzell boundary conditions}
\author[]{Ciprian G. Gal}
\author[]{Gisele Ruiz Goldstein}
\author[]{Jerome A. Goldstein}
\author[]{Silvia Romanelli}
\author[Gal et al.]{Mahamadi Warma}
\address{C.~G.~Gal, Department of Mathematics, University of Missouri,
Columbia, MO 65211 (USA).}
\email{\texttt{galc@missouri.edu}}
\address{G.~Ruiz Goldstein and J.~A.~Goldstein, Department of Mathematics
University of Memphis, Memphis, TN 38152 (USA).}
\email{\texttt{ggoldste@memphis.edu}}
\email{\texttt{jgoldste@memphis.edu}}
\address{S.~Romanelli, Universit degli Studi di Bari Via E. Orabona 4
I-70125 Dipartimento di Matematica Bari, (Italy).}
\email{romans@dm.uniba.it}
\address{M.~Warma, University of Puerto Rico, Department of Mathematics (Rio
Piedras Campus), PO Box 23355 San Juan, PR 00931-3355 (USA).}
\email{\texttt{mjwarma@gmail.com, warma@uprrp.edu}}
\maketitle

\begin{abstract}
We give necessary and sufficient conditions for the solvability of some
semilinear elliptic boundary value problems involving the Laplace operator
with linear and nonlinear highest order boundary conditions involving the
Laplace-Beltrami operator.
\end{abstract}



\section{Introduction}

Let $\Omega \subset \mathbf{R}^{N},$ $N\geq 1,$ be a bounded domain with
smooth boundary $\Gamma :=\partial \Omega $. Let $\alpha :\mathbb{R}%
\rightarrow \mathbb{R}$ be a continuous monotone nondecreasing function with
$\alpha \left( 0\right) =0$ and consider the following boundary value
problem:%
\begin{equation}
\begin{cases}
-\Delta u+\alpha \left( u\right) =f & \text{ in }\Omega , \\
\frac{\partial u}{\partial n}=0 & \text{ on }\Gamma ,%
\end{cases}
\label{1.1}
\end{equation}%
where $f\in L^{2}\left( \Omega \right) $ is a given real function, $\frac{%
\partial u}{\partial n}$ denotes the outward normal derivative of $u$ on $%
\Gamma $ and $\Delta $ is the Laplace operator in $\Omega .$ Let us denote
by $\left\vert \Omega \right\vert $ the Lebesgue measure of $\Omega .$ It is
known that a necessary and sufficient condition for the existence of a
solution of (\ref{1.1}) is%
\begin{equation}
\left\vert \Omega \right\vert ^{-1}\int\limits_{\Omega }f\left( x\right)
dx\in \mathcal{R}\left( \alpha \right) .  \label{1.3}
\end{equation}%
Here all the functions are real valued. This result is due to J. Mawhin \cite%
{Ma}. Earlier, Landesman and Lazer \cite{LL} obtained a similar result. This
result lead to an enormous body of literature. Landesman and Lazer showed
that (\ref{1.3}) is a necessary condition, while a sufficient condition is%
\begin{equation}
\left\vert \Omega \right\vert ^{-1}\int\limits_{\Omega }f\left( x\right)
dx\in int(\mathcal{R}\left( \alpha \right)) ,  \label{1.4}
\end{equation}%
where $int\left( I\right) $ denotes the interior of the set $I$. They also
allowed for nonmonotone $\alpha ,$ which was very important for later
developments. Thus for them, $\alpha :\mathbb{R}\rightarrow \mathbb{R}$ is
continuous, $\alpha \left( 0\right) =0,$ and%
\begin{equation}
\alpha \left( -\infty \right) =\lim_{x\rightarrow -\infty }\alpha \left(
x\right) \leq \alpha \left( y\right) \leq \lim_{x\rightarrow +\infty }\alpha
\left( x\right) =\alpha \left( +\infty \right)  \label{nonlinear}
\end{equation}%
for all $y\in \mathbb{R}$. They proved (\ref{1.3})\ is a necessary condition
in this more general context of (\ref{nonlinear}), while (\ref{1.4}) is a
sufficient condition. Prior to Mawhin's work, Brezis and Haraux \cite{BH}
put the \cite{LL} result in an abstract context and found a new, elegant
proof for it. These works led to very much research, including major
contributions by Brezis and Nirenberg \cite{BN} and many others. Brezis and
Haraux worked in the context of subdifferentials of convex functionals on
Hilbert spaces. We will explain the context and the abstract results, used
in proving the assertion connecting (\ref{1.3}) and (\ref{1.4}), in Sections
2 and 4. But here we emphasize again that these results were inspired by the
similar result of Landesman and Lazer \cite{LL} who, in giving necessary and
sufficient conditions on $f$ for the solvability of certain elliptic
problems of the form $Lu+Nu=f$ (with $L$ linear and $N$ nonlinear),
established a sort of "nonlinear Fredholm alternative" for the first time.
When $\alpha \equiv 0$, the above result reduces to%
\begin{equation*}
-\Delta u=f\text{ \ in\ }\Omega ,\text{ }\frac{\partial u}{\partial n}=0,%
\text{ on }\Gamma
\end{equation*}%
has a weak solution if and only if%
\begin{equation*}
\left\langle f,1\right\rangle _{L^{2}(\Omega )}=0,\text{ i.e}.\text{, }%
\int\limits_{\Omega }f\left( x\right) dx=0,
\end{equation*}%
which is exactly the Fredholm alternative since the null space of the
Neumann Laplacian is the constants. Thus, Mawhin's result (based on the work
in \cite{LL}) is an exact nonlinear Fredholm alternative for the nonlinear
problem (\ref{1.1}).

The goal of this paper is to establish similar results (comparable with (\ref%
{1.3}), (\ref{1.4})) for the following boundary value problem with second
order boundary conditions:%
\begin{equation}
\begin{cases}
-\Delta u+\alpha _{1}\left( u\right) =f\left( x\right) & \text{ in }\Omega ,
\\
b\left( x\right) \frac{\partial u}{\partial n}+c\left( x\right) u-qb\left(
x\right) \Delta _{\Gamma }u+\alpha _{2}\left( u\right) =g\left( x\right) &
\text{ on }\Gamma ,%
\end{cases}
\label{1.7}
\end{equation}%
where the functions appearing in (\ref{1.7}) are real and satisfy $b\in
C\left( \Gamma \right) ,$ $b>0,$ $c\in C\left( \Gamma \right) ,$ $c\geq 0$, $%
q$ is a nonnegative constant; $\alpha _{1},$ $\alpha _{2}:\mathbb{R}%
\rightarrow \mathbb{R}$ are continuous and monotone nondecreasing functions,
such that $\alpha _{i}\left( 0\right) =0$. Above, $\Delta _{\Gamma }$ is the
Laplace-Beltrami operator on $\Gamma $, $f\in L^{2}\left( \Omega \right) $
and $g\in L^{2}\left( \Gamma \right) $ are given real functions. Thus, our
emphasis is on the generality of the boundary conditions.

We organize the paper as follows. In Sections 2 and 3, we discuss the
auxiliary linear problems corresponding to (\ref{1.7}), and in Section 4 we
show the existence of weak solutions to (\ref{1.7}) in case certain global
constraints (similar to (\ref{1.3})) hold. In the same section, we will
consider concrete examples as application of our results.\newline

Before we state our main result, we define the notion of weak solutions to %
\eqref{1.7}.




\begin{definition}
\label{def-weak-sol} A function $u\in H^{1}(\Omega )$ is said to be a weak
solution to \eqref{1.7} if $\alpha _{1}(u)\in L^{1}(\Omega ),$ $\alpha
_{2}(tr(u))\in L^{1}(\Gamma )$, $tr\left( u\right) :=u_{\mid \Gamma }\in
H^{1}(\Gamma ),$ if $q>0$, and%
\begin{eqnarray}
\int_{\Omega }fvdx+\int_{\Gamma }gv\frac{dS}{\beta } &=&\int_{\Omega }\nabla
u\cdot \nabla vdx+\int_{\Omega }\alpha _{1}(u)vdx  \notag
\label{eq-weak-solu} \\
&&+\int_{\Gamma }\left( \alpha _{2}(u)v+cuv\right) \frac{dS}{\beta }%
+q\int_{\Gamma }\nabla _{\Gamma }u\cdot \nabla _{\Gamma }vdS,
\end{eqnarray}%
for all $v\in H^{1}(\Omega )\cap C\left( \overline{\Omega }\right) ,$ if $%
q=0 $ and all $v\in H^{1}(\Omega )\cap C\left( \overline{\Omega }\right) $
with $tr(v)\in H^{1}(\Gamma ),$ if $q>0$.
\end{definition}

Our main result is as follows. Let%
\begin{equation}
\lambda _{1}=\int\limits_{\Omega }dx,\text{ }\lambda
_{2}=\int\limits_{\Gamma }\frac{dS}{b},
\end{equation}%
and let $\widetilde{I}$ be the interval%
\begin{equation*}
\widetilde{I}=\lambda _{1}\mathcal{R}\left( \alpha _{1}\right) +\lambda _{2}%
\mathcal{R}\left( \alpha _{2}\right) .
\end{equation*}%
Moreover, for each $i=1,2$, set%
\begin{equation}
L_{i}(t):=\int_{0}^{t}\alpha _{i}(s)ds\text{ and }\Lambda _{i}(t):=\max
\left\{ L_{i}\left( t\right) ,L_{i}\left( -t\right) \right\} ,\text{ for all
}t\in \mathbb{R}\text{.}  \label{def}
\end{equation}

\begin{theorem}
\label{main} Let $c\equiv 0$ and let $\alpha _{i}:\mathbb{R}\rightarrow
\mathbb{R}$ $(i=1,2)$ be continuous, monotone nondecreasing functions such
that $\alpha _{i}\left( 0\right) =0$. If \eqref{1.7} has a weak solution,
then%
\begin{equation}
\int\limits_{\Omega }f\left( x\right) dx+\int\limits_{\Gamma }g\left(
x\right) \frac{dS}{b\left( x\right) }\in \widetilde{I}.  \label{1.8}
\end{equation}%
Conversely, if there exist positive constants $t_{i},$ $C_{i}>0$, such that
the functions $\Lambda _{i}:\mathbb{R}\rightarrow \lbrack 0,+\infty ),$ $%
i=1,2,$ satisfy $\Lambda _{i}(2t)\leq C_{i}\Lambda _{i}(t),\;$for all $t\geq
t_{i}$, and%
\begin{equation}
\int\limits_{\Omega }f\left( x\right) dx+\int\limits_{\Gamma }g\left(
x\right) \frac{dS}{b\left( x\right) }\in int(\widetilde{I}),  \label{1.8bis}
\end{equation}%
then (\ref{1.7}) has a weak solution.
\end{theorem}

\section{The linear problem}

\quad We need to introduce some notation and terminology. We first define
the space $\mathbb{X}_{2}$ to be the real Hilbert space $L^{2}\left( \Omega
,dx\right) \oplus L^{2}(\Gamma ,dS/b),$ with norm
\begin{equation}
\left\Vert u\right\Vert _{\mathbb{X}_{2}}=\left( \int\limits_{\Omega
}\left\vert u\left( x\right) \right\vert ^{2}dx+\int\limits_{\Gamma
}\left\vert u\left( x\right) \right\vert ^{2}\frac{dS_{x}}{b\left( x\right) }%
\right) ^{\frac{1}{2}}  \label{2.1}
\end{equation}%
for $u\in C\left( \overline{\Omega }\right) $, where $dS$ denotes the usual
Lebesgue surface measure on $\Gamma $. Here, if $u\in C\left( \overline{%
\Omega }\right) ,$ we identify $u$ with the vector $U=\left( u|_{\Omega
},u|_{\Gamma }\right) ^{T}\in C\left( \Omega \right) \times C\left( \Gamma
\right) .$ We then note that $\mathbb{X}_{2}=L^{2}\left( \Omega ,dx\right)
\oplus L^{2}(\Gamma ,dS/b)$ is the completion of $C\left( \overline{\Omega }%
\right) $ with respect to the norm $\left( 2.1\right) $. In general, any
vector $U\in \mathbb{X}_{2}$ will be of the form $\left( u_{1},u_{2}\right)
^{T}$ with $u_{1}\in L^{2}\left( \Omega ,dx\right) $ and $u_{2}\in
L^{2}(\Gamma ,dS/b),$ and there need be no connection between $u_{1}$ and $%
u_{2}.$ Here and below the superscript $T$ denotes transpose. Let $%
\left\langle \cdot ,\cdot \right\rangle _{\mathbb{X}_{2}}$ denote the
corresponding inner product on $\mathbb{X}_{2}$. For a complete discussion
of this space, we refer the reader to \cite{FGGR1}$.$

We define the formal operator $A_{0}$ by%
\begin{equation}
A_{0}U=\left( \left( -\Delta u\right) |_{\Omega },\left( -\Delta u\right)
|_{\Gamma }\right) ^{T},  \label{2.2}
\end{equation}%
for functions $U=\left( u|_{\Omega },u|_{\Gamma }\right) ^{T}$ with $u\in
C^{2}\left( \overline{\Omega }\right) $ that satisfy the Wentzell boundary
condition%
\begin{equation}
\,\Delta u+b\left( x\right) \frac{\partial u}{\partial n}+c\left( x\right)
u-qb\left( x\right) \Delta _{\Gamma }u=0,  \label{2.3}
\end{equation}%
on $\Gamma .$ Here $\left( \Delta u\right) _{|_{\Gamma }}$ stands for the
trace of the function $\Delta u$ on the boundary $\Gamma $ and it should not
be confused with the Laplace-Beltrami operator $\Delta _{\Gamma }u$. From
now on, $tr\left( u\right) $ denotes the trace of $u$ on the boundary. We let%
\begin{align}
D\left( A_{0}\right) & =\left\{ U=\left( u_{1},u_{2}\right) ^{T}\in \mathbb{X%
}_{2}:U\text{ corresponds to }u_{1}\in C^{2}\left( \overline{\Omega }\right)
,\right.  \label{2.4} \\
& \left. u_{2}=u_{1}|_{\Gamma }=tr\left( u_{1}\right) \text{ and (\ref{2.3})
holds}\right\} .  \notag
\end{align}%
For functions $u\in C^{2}\left( \overline{\Omega }\right) \subset \mathbb{X}%
_{2}$, $A_{0}U$ is defined by (\ref{2.2}). For any functions $u,v$ belonging
to $C^{2}\left( \overline{\Omega }\right) ,$ and each satisfying the
boundary condition $\Delta \varpi +b\left( x\right) \frac{\partial \varpi }{%
\partial n}+c\left( x\right) \varpi -qb\left( x\right) \Delta _{\Gamma
}\varpi =0$ on $\Gamma ,$ we identify $u$ and $v$ with $U=\left( u|_{\Omega
},u|_{\Gamma }\right) ^{T}$ and $V=\left( v|_{\Omega },v|_{\Gamma }\right)
^{T}$ and calculate $\left\langle A_{0}U,V\right\rangle _{\mathbb{X}_{2}}$
as follows:%
\begin{align}
\left\langle A_{0}U,V\right\rangle _{\mathbb{X}_{2}}& =\int\limits_{\Omega
}\left( -\Delta u\right) vdx+\int\limits_{\Gamma }\left( -\Delta u\right) v%
\frac{dS}{b\left( x\right) }  \label{2.5} \\
& =\int\limits_{\Omega }\nabla u\cdot \nabla vdx+\int\limits_{\Gamma }\left(
-\Delta u-b\left( x\right) \frac{\partial u}{\partial n}\right) v\frac{dS}{%
b\left( x\right) }  \notag \\
& =\int\limits_{\Omega }\nabla u\cdot \nabla vdx+\int\limits_{\Gamma }\left(
c\left( x\right) u-qb\left( x\right) \Delta _{\Gamma }u\right) v\frac{dS}{%
b\left( x\right) },  \notag
\end{align}%
since $-\Delta u-b\left( x\right) \frac{\partial u}{\partial n}=c\left(
x\right) u-qb\left( x\right) \Delta _{\Gamma }u$ on $\Gamma .$ Furthermore,
Stokes' theorem applied in the last term of (\ref{2.5}) yields%
\begin{equation}
\left\langle A_{0}U,V\right\rangle _{\mathbb{X}_{2}}=\int\limits_{\Omega
}\nabla u\cdot \nabla vdx+\int\limits_{\Gamma }c\left( x\right) uv\frac{dS}{%
b\left( x\right) }+q\int\limits_{\Gamma }\nabla _{\Gamma }u\cdot \nabla
_{\Gamma }vdS,  \label{2.6}
\end{equation}%
where $\nabla _{\Gamma }$ stands for the tangential gradient on the surface $%
\Gamma .$ Finally, if we denote the right hand side of (\ref{2.6}) by $%
\varrho \left( U,V\right) $, it is now clear that $\varrho \left( U,V\right)
=\varrho \left( V,U\right) =\left\langle U,A_{0}V\right\rangle _{\mathbb{X}%
_{2}},$ therefore $A_{0}$ is symmetric on $\mathbb{X}_{2}$. Let us now
consider a function $f\in C\left( \overline{\Omega }\right) \cup H^{1}\left(
\Omega \right) $ such that $F=\left( f_{1},f_{2}\right) ^{T}$ with $%
f_{1}:=f|_{\Omega }$ and $f_{2}:=f|_{\Gamma }.$ By the equality $A_{0}U=F,$
we mean the following boundary value problem:%
\begin{equation}
-\Delta u=f_{1}\quad \text{ in }\quad \quad \Omega ,  \label{2.7}
\end{equation}%
\begin{equation}
-\Delta u=f_{2}\,\quad \text{on}\quad \quad \Gamma .  \label{2.8}
\end{equation}%
Using the Wentzell boundary condition (\ref{2.3}) and replacing $f_{2}$ by $%
f_{\mid \Gamma },$ the boundary condition (\ref{2.8}) becomes%
\begin{equation}
b(x)\frac{\partial u}{\partial n}+c(x)u-qb(x)\Delta _{\Gamma }u=f_{2}\text{
on }\Gamma .  \label{2.9}
\end{equation}%
Any $u\in H^{s}\left( \Omega \right) $ has a trace $tr\left( u\right)
=u|_{\Gamma }$ in $H^{s-1/2}\left( \Gamma \right) $ for $s>1/2.$ More
precisely, we recall that the linear map $tr:H^{s}\left( \Omega \right)
\rightarrow H^{s-1/2}\left( \Gamma \right) $ is bounded and onto for $s>1/2$%
. We now define the "Wentzell version of $A_{0}$", $\widetilde{A}_{0},$ by $%
\widetilde{A}_{0}U=F=\left( f_{1},f_{2}\right) ^{T}$ on%
\begin{align}
D(\widetilde{A}_{0})& =\left\{ U\in \mathbb{X}_{2}:U\text{ corresponds to }%
u\in H^{2}\left( \Omega \right) ,\right.  \label{2.10} \\
& \left. tr\left( u\right) \in H^{2}\left( \Gamma \right) \text{ if }q>0%
\text{, and (\ref{2.7}), (\ref{2.9}) holds}\right\} .  \notag
\end{align}%
In this case, $f_{2}$ need not be the trace of $f_{1}$ on $\Gamma .$ Then,
using the techniques as in \cite{FGGR2}, we can easily check that $%
\widetilde{A}_{0}$ is contained in the closure of $A_{0}$. Let $A=\overline{A%
}_{0}=\overline{\widetilde{A}}_{0}$. Then, $A$ is selfadjoint and
nonnegative if $c\geq 0$ on $\Gamma $; $A$ is the operator associated with
the nonnegative symmetric closed bilinear form $\varrho \left( U,V\right) .$
We have $\left\langle AU,V\right\rangle _{\mathbb{X}_{2}}=\varrho \left(
U,V\right) ,$ for all $U\in D\left( A\right) $ and all $V=\left( v|_{\Omega
},v|_{\Gamma }\right) ^{T}\in D(\varrho ):=H^{1}\left( \Omega \right) \times
H^{1}\left( \Gamma \right) $ (if $q>0$) and $V\in D(\varrho ):=H^{1}\left(
\Omega \right) \times H^{1/2}\left( \Gamma \right) $ if $q=0$. We emphasize
that for $A=\overline{A}_{0},$ the equations (\ref{2.7}) and (\ref{2.9})
hold even if the vector $F=\left( f_{1},f_{2}\right) ^{T}$ does not
correspond to a function $f$ belonging to $C\left( \overline{\Omega }\right)
\cup H^{1}\left( \Omega \right) $, that is, $f_{2}\neq f_{1}|_{\Gamma }.$
For $U\in D\left( A\right) ,$ an operator matrix representation of $A$ is
given by%
\begin{equation}
A=\left(
\begin{array}{cc}
-\Delta & \qquad 0 \\
b\frac{\partial }{\partial n} & \qquad cI-qb\Delta _{\Gamma }%
\end{array}%
\right) .  \label{2.12}
\end{equation}

We will now give a concrete example when $q=0$ (that is, $\Delta _{\Gamma }$
does not appear in the boundary condition (\ref{2.9})). This is a simple
example where $f_{2}\neq f_{1}|_{\Gamma }$.

\begin{example}
Let $\Omega =\left( 0,1\right) \subset \mathbb{R}$ and let $F=\left(
0,k\right) $ where $\Gamma =\left\{ 0,1\right\} $ and $k\left( 0\right)
=a_{0},$ $k\left( 1\right) =b_{0}$ with $\left( a_{0},b_{0}\right) \neq
\left( 0,0\right) $. Take $b\left( j\right) =c\left( j\right) =1$ for $%
j=0,1. $ Then $AU=F$ means%
\begin{equation}
\left\{
\begin{array}{c}
u^{^{\prime \prime }}=0\text{ in }\left[ 0,1\right] , \\
-u^{^{\prime }}\left( 0\right) +u\left( 0\right) =a_{0}, \\
u^{^{\prime }}\left( 1\right) +u\left( 1\right) =b_{0},%
\end{array}%
\right.  \label{2.13}
\end{equation}%
since $\partial /\partial n=\left( -1\right) ^{j+1}d/dx$ at $x=j\in \left\{
0,1\right\} $. Solving (\ref{2.13}) gives%
\begin{equation*}
u\left( x\right) =\frac{1}{3}\left[ \left( b_{0}-a_{0}\right) x+\left(
2a_{0}+b_{0}\right) \right] ,\text{ }x\in \left[ 0,1\right] .
\end{equation*}
\end{example}

\section{The domain of the Wentzell Laplacian}

We recall some facts from the theory of linear elliptic boundary value
problems. The standard theory works for uniformly elliptic problems of even
order $2m$; we shall restrict ourselves to the second order case, $m=1$. We
shall treat the symmetric case, although this restriction is not needed for
the results we present in this section. Our problem takes the form $\widehat{%
A}u=f$ in $\Omega ,$ $\widehat{B}u=g$ on $\Gamma $, where $\Omega $ is a
smooth bounded domain in $\mathbb{R}^{N}$ with boundary $\Gamma ,$%
\begin{equation}
\left\{
\begin{array}{c}
Au=-\nabla \cdot \mathcal{A}\left( x\right) \nabla u, \\
Bu=b\partial _{n}^{\mathcal{A}}u+cu-qb\Delta _{\Gamma }u,%
\end{array}%
\right.  \label{3.0}
\end{equation}%
and $\widehat{A}=A+\lambda I,$ $\widehat{B}=B+\lambda I,$ for some $\lambda
\in \mathbb{R}$. As the theory is based upon pseudo differential operator
techniques, we make the standard assumption that $\Omega $, $\mathcal{A}$, $%
b $ and $c$ are all of class $C^{\infty }$ in addition to the assumptions
that the $N\times N$ matrix function $\mathcal{A}$ is real, symmetric and
uniformly positive definite, $b>0$, $c\geq 0$ and $q\in \left[ 0,+\infty
\right) $.

Let $s\in \mathbb{N}_{0}=\left\{ 0,1,2,...\right\} $ and $p\in \left(
1,+\infty \right) $. We refer to Triebel \cite{T} for the general case,
where we use his notation; Lions-Magenes \cite{LM} treats the Hilbert space
case ($p=2$).

\begin{theorem}
\label{t3.1}Let the above assumptions hold, with $q=0$. Then for all $%
\lambda >0$, with $\widehat{A}=A+\lambda I,$ $\widehat{B}=B+\lambda I$, the
map%
\begin{equation*}
\Xi _{\lambda }:u\mapsto (\widehat{A}u,\widehat{B}u),
\end{equation*}%
viewed as a map from $W_{p}^{s+2}\left( \Omega \right) $ to $W_{p}^{2}\left(
\Omega \right) \times B_{p,p}^{1+s-1/p}\left( \Gamma \right) ,$ is an
isomorphism.
\end{theorem}

This means that $\Xi _{\lambda }$ is a linear bijection, and there is a
positive constant $C$, independent of $u$, such that%
\begin{equation}
C^{-1}\left\Vert u\right\Vert _{W_{p}^{s+2}\left( \Omega \right) }\leq
\left\Vert \widehat{A}u\right\Vert _{W_{p}^{2}\left( \Omega \right)
}+\left\Vert \widehat{B}u\right\Vert _{B_{p,p}^{1+s-1/p}\left( \Gamma
\right) }\leq C\left\Vert u\right\Vert _{W_{p}^{s+2}\left( \Omega \right) }
\label{3.0bis}
\end{equation}%
for all $u\in W_{p}^{s+2}\left( \Omega \right) $. Thus, the isomorphism is a
linear homeomorphism, but it not need be isometric. Here $W_{p}^{r}\left(
\Omega \right) $ is Triebel's notation for the Sobolev space and $%
B_{p,p}^{r}\left( \Gamma \right) $ for the Besov space. For $s=0$ and $p=2,$
this reduces $\Xi _{\lambda }$ to being an isomorphism from $H^{2}\left(
\Omega \right) $ to $L^{2}\left( \Omega \right) \oplus L^{2}\left( \Gamma
,dS\right) ,$ which is equivalent to saying that $\Xi _{\lambda }$ is an
isomorphism from $H^{2}\left( \Omega \right) $ to $\mathbb{X}_{2}$, since $%
L^{2}\left( \Gamma ,dS\right) $ and $L^{2}\left( \Gamma ,dS/b\right) $ are
the same sets with equivalent inner products.

It follows that, when $q=0$, the domain of the Wentzell Laplacian $A$ is
exactly $H^{2}\left( \Omega \right) $.

\begin{theorem}
\label{t3.2}Let
\begin{equation*}
H_{\ast }^{2}\left( \Omega \right) =\left\{ u\in H^{2}\left( \Omega \right)
:u_{\mid \Gamma }\in H^{2}\left( \Gamma \right) \right\} .
\end{equation*}%
The domain of the Wentzell Laplacian $A$, the selfadjoint closure of $A_{0}$%
, defined by (\ref{2.7}), (\ref{2.9}), is exactly%
\begin{equation}
D\left( A\right) =\left\{
\begin{array}{cc}
H^{2}\left( \Omega \right) & \text{if }q=0, \\
H_{\ast }^{2}\left( \Omega \right) & \text{if }q>0.%
\end{array}%
\right.  \label{3.1}
\end{equation}%
The same conclusion holds for the closure of the operator $A$ defined by (%
\ref{3.0}).
\end{theorem}

Before outlining the proof of this theorem we make some remarks. Theorem \ref%
{t3.2} gives the first "simple" explicit characterization of $D\left(
A\right) $, including the case of $q>0.$ Normally, knowing that $D\left(
A_{0}\right) $ is a core for $A$ is enough for most purposes involving
linear problems. But we need to know $D\left( A\right) $ exactly in order to
apply the Brezis-Haraux result (see Proposition \ref{theo-equ} below).
Theorem \ref{t3.2} assumes that $\Gamma ,$ $b$ and $c$ are $C^{\infty }$.
Surely this much regularity is not needed. But the proof is based on pseudo
differential operator techniques and this theory is always presented in the $%
C^{\infty }$ context, because to do otherwise would entail many complicated
calculations requiring a lot of courage. So Theorem \ref{t3.2} should be
valid if everything is $C^{2},$ but this is merely an educated guess
(however, see Remark 3.1).

We wish to recall the earlier work on this problem by Escher \cite{Es} (see
also Fila and Quittner \cite{FQ}). Escher proved Theorem \ref{t3.2} in the
special case of $b\equiv 1$ and $q=0$. He worked in the $\mathbb{X}_{p}$
context for $1<p<+\infty ,$ but, by focusing on the analytic semigroup
aspect of the problem, he did not notice the selfadjointness of $A$.
Moreover, his restriction to the case of $b\equiv 1$ avoids many interesting
cases, since the coefficient $b$ has physical significance (cf. \cite{Gis}).

We now recall the strategy of the proof of Theorem \ref{t3.1}. We outline
the proof in several steps:\newline

\textbf{Step 1}. Treat the case of constant coefficients and take $\Omega $
to be a half-space.\newline

\textbf{Step 2}. Then localizing and using a partition of unity, this breaks
the problem down into a large number of problems $\left\{ P_{j}\right\} ,$
where the portion of $\Gamma $ is the subdomain corresponding to $P_{j}$ is
almost flat and the coefficients are almost constants.\newline

Flatten out the boundary and solve each $P_{j},$ using Step 1, and the
theory of pseudo differential operators (see, e.g., Taylor \cite{Ta}).
Finally, put everything together and complete the proof. The proof is quite
long, technical and complicated, but it is now well understood and standard.
For the moment we focus on Step 1 and, for simplicity, assume that $\mathcal{%
A}$ is the identity matrix, so that $A=-\Delta $. Then our problem (\ref{3.0}%
) becomes the constant coefficient problem:%
\begin{equation}
\left\{
\begin{array}{cc}
\widehat{A}u=-\Delta u+\lambda u=f & \text{in }\mathbb{R}_{+}^{N}, \\
\widehat{B}u=b\partial _{n}u+cu+\lambda u-qb\Delta _{\Gamma }u=g & \text{on }%
\partial \mathbb{R}_{+}^{N}.%
\end{array}%
\right.  \label{3.2}
\end{equation}%
Here $\mathbb{R}_{+}^{N}=\{x=\left( y,z\right) :y\in \mathbb{R}^{N-1},$ $%
z\geq 0\}$, $\partial \mathbb{R}_{+}^{N}=\{x=\left( y,0\right) :y\in \mathbb{%
R}^{N-1}\}$ and the boundary condition of (\ref{3.2}) is equivalent to%
\begin{equation}
b\partial _{z}u+cu+\lambda u-qb\Delta _{y}u=g  \label{3.4}
\end{equation}%
on $\partial \mathbb{R}_{+}^{N}.$ For a function $h\left( y,z\right) ,$ let $%
\widehat{h}\left( \zeta ,z\right) $ be the Fourier transform in the $\mathbb{%
R}^{N-1}$-variable with $z$ fixed:%
\begin{equation*}
\widehat{h}\left( \zeta ,z\right) =\left( 2\pi \right) ^{\frac{1-N}{2}%
}\int\limits_{\mathbb{R}^{N-1}}e^{-i\zeta \cdot y}h\left( y,z\right) dy,%
\text{ }\left( \zeta ,z\right) \in \mathbb{R}_{+}^{N}.
\end{equation*}%
Then, in Fourier space, the first equation of (\ref{3.2}) and equation (\ref%
{3.4}) become%
\begin{equation}
\frac{\partial ^{2}\widehat{u}}{\partial z^{2}}-\left( \left\vert \zeta
\right\vert ^{2}+\lambda \right) \widehat{u}=\widehat{f}\text{ in }\mathbb{R}%
_{+}^{N},  \label{3.5}
\end{equation}%
\begin{equation}
b\frac{\partial \widehat{u}}{\partial z}+\left( c+\lambda +qb\left\vert
\zeta \right\vert ^{2}\right) \widehat{u}=\widehat{g}\text{ on }\partial
\mathbb{R}_{+}^{N}.  \label{3.6}
\end{equation}%
We need $u$ to be an $L^{2}$ function. To solve (\ref{3.5}), one finds the
general solution of the homogeneous equation and adds to it a particular
solution of (\ref{3.5}), obtained by the variation of constants formula. The
general solution of the homogenous version of (\ref{3.5}) is%
\begin{equation}
\widehat{u}\left( \zeta ,z\right) =C_{1}e^{\gamma _{1}z}+C_{2}e^{\gamma
_{2}z},  \label{3.6bis}
\end{equation}%
where%
\begin{equation*}
\gamma _{j}=\left( -1\right) ^{j+1}\left( \left\vert \zeta \right\vert
^{2}+\lambda \right) ^{1/2},\text{ }j=1,2.
\end{equation*}%
Then for each $\zeta \in \mathbb{R}^{N-1}$, $\gamma _{2}<0<\gamma _{1}.$
Thus, the general $L^{2}$ solution of the homogeneous problem is given by (%
\ref{3.6bis}), with $C_{2}$ an arbitrary constant and $C_{1}=0$.

Next, (\ref{3.6}) is of the form%
\begin{equation*}
\frac{\partial \widehat{u}}{\partial z}-p\left( \zeta \right) =m\left( \zeta
\right) ,
\end{equation*}%
for $z=0,$ where $p\geq \varepsilon _{0}>0$ for all $\zeta $ (For more
general problems, the corresponding inequality follows from uniform
ellipticity). It follows that (\ref{3.2}) (as well as (\ref{3.5}), (\ref{3.6}%
)) has a unique $L^{2}$ solution. Note that this works for $q>0$ as well as
for $q=0$. For $q>0,$ we require that $\left\vert \zeta \right\vert ^{2}%
\widehat{u}$ as well as $\widehat{u}$ is in $L^{2}$. If one studies the
proof in \cite{T} in detail, minor modifications of the tedious calculations
lead to the proof of Theorem \ref{t3.2}.

More precisely, for $q>0$, we conclude that there is a positive constant $%
C=C\left( q,b,c,\lambda ,\mathcal{A}\right) ,$ for every $\lambda >0,$ such
that%
\begin{equation}
C^{-1}\left\Vert u\right\Vert _{H_{\ast }^{2}\left( \Omega \right) }\leq
\left\Vert \left( \widehat{A}u,\widehat{B}u\right) ^{T}\right\Vert _{\mathbb{%
X}_{2}}\leq C\left\Vert u\right\Vert _{H_{\ast }^{2}\left( \Omega \right) }
\label{estq>0}
\end{equation}%
for all $u\in H_{\ast }^{2}\left( \Omega \right) .$ Moreover, the map $%
u\mapsto \left( \widehat{A}u,\widehat{B}u\right) ^{T}$ is a surjective
linear isomorphism of $H_{\ast }^{2}\left( \Omega \right) $ onto $\mathbb{X}%
_{2},$ for $q>0$. Above in (\ref{estq>0}), the norm in $H_{\ast }^{2}\left(
\Omega \right) $ is defined as%
\begin{equation*}
\left\Vert u\right\Vert _{H_{\ast }^{2}\left( \Omega \right) }=\left(
\left\Vert u\right\Vert _{H^{2}\left( \Omega \right) }^{2}+\left\Vert
tr\left( u\right) \right\Vert _{H^{2}\left( \Gamma \right) }^{2}\right)
^{1/2}.
\end{equation*}%
From this, the proof of Theorem \ref{t3.2} follows. $\square $\newline

\noindent \textbf{Remark 3.1. }We note that the first inequality of (\ref%
{estq>0}) was already obtained in \cite[Lemma A.1]{MZ}\ for the weak
solutions of (\ref{3.0}), using standard Sobolev inequalities and assuming $%
b,$ $c\in C\left( \Gamma \right) ,$ $b,\lambda >0,$ $\mathcal{A}=I_{N\times
N}$ and $\Gamma $ is of class $C^{2}$. Observe that (\ref{3.0}) is also an
elliptic boundary value problem in the sense specified in \cite{H, Pe},
where similar estimates to (\ref{estq>0}) were also obtained. The second
inequality of (\ref{estq>0}) is obvious and is based on the definition of $%
\widehat{A}$ and $\widehat{B}$.

\section{Convex analysis}

We begin with the following assumptions:\newline

\textbf{(H1) }The functions $\alpha _{i}:\mathbb{R}\rightarrow \mathbb{R}$, $%
i=1,2,$ are continuous, monotone nondecreasing with $\alpha _{i}(0)=0$.%
\newline

\textbf{(H2) }Let $\Lambda _{i}$ be as in (\ref{def}) and suppose that they
satisfy the \emph{$\triangle _{2}$-condition near infinity,} in the sense
that, there are positive constants $t_{i},$ $C_{i}>0$, $i=1,2,$ such that%
\begin{equation}
\Lambda _{i}(2t)\leq C_{i}\Lambda _{i}(t),\;\mbox{ for all }\;t\geq t_{i}.
\label{delta}
\end{equation}

Let $\tilde{\alpha}_{i}:\;\mathbb{R}\rightarrow \mathbb{R}$ $(i=1,2)$ be the
inverse of $\alpha _{i}$. Then $\tilde{\alpha}_{i}$ is a nondecresing
function from $\mathbb{R}$ to $\mathbb{R}$, which is multivalued at its
jumps and it is in $L_{loc}^{1}\left( \mathbb{R}\right) $. Its graph is a
connected subset of $\mathbb{R}^{2}$. Let $\widetilde{L}_{i}:\mathbb{R}%
\rightarrow \lbrack 0,+\infty ),$ $i=1,2,$ be defined by%
\begin{equation}
\widetilde{L}_{i}(t):=\int_{0}^{t}\widetilde{\alpha }_{i}(s)ds\text{ and }%
\widetilde{\Lambda }_{i}:=\max \left\{ \widetilde{L}_{i}(t),\widetilde{L}%
_{i}(-t)\right\} ,\text{ for all }t\in \mathbb{R}.  \label{lambda-j}
\end{equation}

All the functions given in (\ref{def}) and (\ref{lambda-j}) are convex and
continuous on $\mathbb{R}$, nondecreasing on $\mathbb{R}_{+}$, and all
vanish at the origin; $\Lambda _{i}$ and $\widetilde{\Lambda }_{i}$ are even
functions and are complementary Young functions in the sense of \cite[Chap.
I, Section 1.3, Theorem 3]{RR-1}, but they need not be $N$-functions. Note
that $L_{i}^{^{\prime }}\left( t\right) =\alpha _{i}\left( t\right) $ on $%
\mathbb{R}$ and $\widetilde{L}_{i}^{^{\prime }}\left( t\right) =\widetilde{%
\alpha }_{i}\left( t\right) $ a.e.; $\left\vert \Lambda _{i}^{^{\prime
}}\left( t\right) \right\vert \geq \left\vert \alpha _{i}\left( t\right)
\right\vert $ and $\left\vert \widetilde{\Lambda }_{i}^{^{\prime }}\left(
t\right) \right\vert \geq \left\vert \widetilde{\alpha }_{i}\left( t\right)
\right\vert $ almost everywhere. It follows, from \cite[Chap. I, Section
1.3, Theorem 3]{RR-1}, that for all $s,$ $t\in \mathbb{R}$,
\begin{equation}
\left\vert st\right\vert \leq L_i(t)+\widetilde L_i(s)\leq \Lambda _{i}(t)+%
\widetilde{\Lambda }_{i}(s).  \label{RR-ine}
\end{equation}%
Suppose that $\Lambda _{i}\left( s\right) =L_{i}\left( \tau \right) $ and $%
\widetilde{\Lambda }_{i}\left( s\right) =\widetilde{L}_{i}\left( \sigma
\right) $, where $\tau $ is $s$ or $-s$ and $\sigma $ is $t$ or $-t$. If $%
\tau =\widetilde{\alpha }_{i}(\sigma )$ or $\sigma =\alpha _{i}(\tau ),$
then we also have equality, that is,%
\begin{equation}
\widetilde L_i(\alpha_i(\tau))=\widetilde{\Lambda }_{i}(\alpha _{i}(\tau
))=\tau \alpha _{i}(\tau )-\Lambda _{i}(\tau
)=\tau\alpha_i(\tau)-L_i(\tau),\;\text{ }i=1,2.  \label{RR-ega}
\end{equation}

Let now $\alpha _{i}:\;\mathbb{R}\rightarrow \mathbb{R}$, $i=1,2,$ satisfy
\textbf{(H1)}. Define the functional $J:\mathbb{X}_{2}\rightarrow \lbrack
0,+\infty ]$ by%
\begin{equation}
J\left( U\right) =\frac{1}{2}\int\limits_{\Omega }\left\vert \nabla
u\right\vert ^{2}dx+\int\limits_{\Omega }L_{1}\left( u\right)
dx+\int\limits_{\Gamma }j_{2}\left( x,u\right) \frac{dS}{b\left( x\right) },
\label{4.1}
\end{equation}%
for $U=\left( u,tr\left( u\right) \right) ^{T},$ $u\in H^{1}\left( \Omega
\right) $ such that all three integrals exist, and $tr\left( u\right) \in
H^{1}\left( \Gamma \right) $ if $q>0$. We take%
\begin{equation}
j_{2}\left( x,u\right) =c\left( x\right) \frac{u^{2}}{2}+qb\left( x\right)
\frac{\left\vert \nabla _{\Gamma }u\right\vert ^{2}}{2}+L_{2}\left( u\right)
.  \label{4.2}
\end{equation}%
The effective domain $\mathbb{D}_{q}:=D\left( J\right) $\ of the functional $%
J$ is precisely%
\begin{equation}
\mathbb{D}_{0}=\{U=\left( u,tr\left( u\right) )\right) ^{T}:u\in H^{1}\left(
\Omega \right) ,\;\int_{\Omega }\Lambda _{1}(u)dx+\int_{\Gamma }\Lambda
_{2}(u)\frac{dS}{b\left( x\right) }<\infty \}  \label{efec-0}
\end{equation}%
if $q=0$, and
\begin{equation}
\mathbb{D}_{q}=\{U=\left( u,tr\left( u\right) )\right) ^{T}\in \mathbb{D}%
_{0}:tr\left( u\right) \in H^{1}\left( \Gamma \right) \}  \label{efec-1}
\end{equation}%
if $q>0$, respectively. Define $J\left( U\right) =+\infty ,$ for all $U\in
\mathbb{X}_{2}\backslash \mathbb{D}_{q},$ $q\geq 0$. As before, for $u\in
H^{1}\left( \Omega \right) $, we identify $u$ with $U=$ $\left( u,tr\left(
u\right) \right) ^{T}\in \mathbb{X}_{2}$. Then $J$ is proper, convex and
lower semicontinuous on $\mathbb{X}_{2},$ as can be shown adapting the ideas
of Brezis \cite{B} (see also \cite{FGGR2}).


Suppose now that $\alpha _{i}$, $i=1,2$, satisfies assumptions \textbf{(H1)}-%
\textbf{(H2)}. Then, by (\ref{delta}), the monotonicity (on $\mathbb{R}_{+}$%
) and the convexity of $\Lambda _{i},$\ $i=1,2,$ we have that $\mathbb{D}%
_{q},$ $q\geq 0,$ is a vector space (see, e.g., \cite[Chap. III, Section
3.1, Theorem 2]{RR-1}).

In what follows, we shall compute the subdifferential of $J$. To this end,
let $F:=(f,g)^{T}\in \mathbb{X}_{2}$ and $U=(u,tr(u))^{T}\in \mathbb{D}_{q}$%
. We claim that $F\in \partial J(U)$ if and only if%
\begin{align}
-\Delta u+\alpha _{1}(u)& =f\;\text{in }\mathcal{D}^{\prime }(\Omega ),
\label{subdiff} \\
b(x)\frac{\partial u}{\partial n}+c(x)u-qb(x)\Delta _{\Gamma }u+\alpha
_{2}(u)& =g\text{ on }\Gamma .  \notag
\end{align}%
First, assume that $F\in \partial J(U)$. Then, by definition, for every $%
V=(v,tr(v))^{T}\in \mathbb{D}_{q}$, we have%
\begin{align}
\int_{\Omega }f(v-u)dx+\int_{\Gamma }g(v-u)\frac{dS}{b}& \leq \frac{1}{2}%
\int_{\Omega }\left( |\nabla v|^{2}-|\nabla u|^{2}\right) dx
\label{subdiff2} \\
& +\int_{\Omega }\left( L_{1}(v)-L_{1}(u)\right) dx+\int_{\Gamma }\left(
j_{2}(x,v)-j_{2}(x,u)\right) \frac{dS}{b},  \notag
\end{align}%
where, from (\ref{4.2}), we find that%
\begin{align*}
\int_{\Gamma }\left( j_{2}(x,v)-j_{2}(x,u)\right) \frac{dS}{b}& =\frac{1}{2}%
\int_{\Gamma }c\left( |v|^{2}-|u|^{2}\right) \frac{dS}{b}+q\frac{1}{2}%
\int_{\Gamma }\left( |\nabla _{\Gamma }v|^{2}-|\nabla _{\Gamma
}u|^{2}\right) dS \\
& +\int_{\Gamma }\left( L_{2}(v)-L_{2}(u)\right) \frac{dS}{b}.
\end{align*}%
Let $W=(w,tr(w))^{T}\in \mathbb{D}_{q}$ be fixed and let $t\in \lbrack 0,1]$%
. Choosing $V:=tW+(1-t)U\in \mathbb{D}_{q}$ in (\ref{subdiff2}), dividing by
$t$ and taking the limit as $t\rightarrow 0^{+},$ from (\ref{subdiff2}), we
obtain%
\begin{align}
& \int_{\Omega }f(w-u)dx+\int_{\Gamma }g(w-u)\frac{dS}{b}  \label{subdiff3}
\\
& \leq \int_{\Omega }\nabla u\cdot \nabla (w-u)dx+\int_{\Omega }\alpha
_{1}(u)(w-u)dx+\int_{\Gamma }cu(w-u)\frac{dS}{b}  \notag \\
& +q\int_{\Gamma }\nabla _{\Gamma }u\cdot \nabla _{\Gamma
}(w-u)dS+\int_{\Gamma }\alpha _{2}(u)(w-u)\frac{dS}{b}.  \notag
\end{align}%
Here we used the definition of the functions $L_{i}$ ($i=1,2$) from (\ref%
{def}) and the Lebesgue Dominated convergence theorem, which implies%
\begin{equation*}
\lim_{t\rightarrow 0^{+}}\int_{\Omega }\frac{L_{1}(u+t(w-u))-L_{1}(u)}{t}%
dx=\int_{\Omega }\alpha _{1}(u)(w-u)dx
\end{equation*}%
and
\begin{equation*}
\lim_{t\rightarrow 0^{+}}\int_{\Gamma }\frac{L_{2}(u+t(w-u))-L_{2}(u)}{t}%
\frac{dS}{b}=\int_{\Gamma }\alpha _{2}(u)(w-u)\frac{dS}{b}.
\end{equation*}%
Letting $W=U\pm \Psi $ in (\ref{subdiff3}), where $\Psi =(\psi ,tr(\psi
))^{T}$ is an arbitrary element of $\mathbb{D}_{q}$, we easily deduce%
\begin{align}
\int_{\Omega }f\psi dx+\int_{\Gamma }g\psi \frac{dS}{b}& =\int_{\Omega
}\nabla u\cdot \nabla \psi dx+\int_{\Omega }\alpha _{1}(u)\psi dx
\label{subdiff4} \\
& +\int_{\Gamma }cu\psi \frac{dS}{b}+q\int_{\Gamma }\nabla _{\Gamma }u\cdot
\nabla _{\Gamma }\psi \frac{dS}{b}+\int_{\Gamma }\alpha _{2}(u)\psi \frac{dS%
}{b}.  \notag
\end{align}%
Taking $\psi \in C_{0}^{\infty }(\Omega )$ in (\ref{subdiff4}), one obtains
the first equation of (\ref{subdiff}). A simple partial integration argument
shows that one also has the second equation in (\ref{subdiff4}).

We shall now prove the converse. Let $U=(u,tr(u))^{T}\in \mathbb{D}_{q}$ be
fixed and let $V=(v,tr(v))^{T}\in \mathbb{D}_{q}$ be arbitrary. On account
of \eqref{RR-ine} and \eqref{RR-ega}, we have%
\begin{align}
\alpha _{1}(u)(v-u)& =\alpha _{1}(u)v-\alpha _{1}(u)u  \label{eq4-18-1} \\
& \leq L_{1}(v)+\widetilde{L}_{1}(\alpha _{1}(u))-\alpha _{1}(u)u\text{,}
\notag \\
& \leq L_{1}(v)-L_{1}(u)  \notag
\end{align}%
and
\begin{align}
\alpha _{2}(u)(v-u)& =\alpha _{2}(u)v-\alpha _{2}(u)u  \label{eq4-19-2} \\
& \leq L_{2}(v)+\widetilde{L}_{2}(\alpha _{2}(u))-\alpha _{2}(u)u\text{,}
\notag \\
& \leq L_{2}(v)-L_{2}(u).  \notag
\end{align}%
Therefore, by (\ref{4.1}) and using (\ref{eq4-18-1})-(\ref{eq4-19-2}), we
have%
\begin{align}
J(V)-J(U)& =\frac{1}{2}\int_{\Omega }\left( |\nabla v|^{2}-|\nabla
u|^{2}\right) dx+\int_{\Omega }\left( L_{1}(v)-L_{1}(u)\right) dx
\label{subdiff5} \\
& +\frac{1}{2}\int_{\Gamma }c\left( |v|^{2}-|u|^{2}\right) \frac{dS}{b}+%
\frac{q}{2}\int_{\Gamma }\left( |\nabla _{\Gamma }v|^{2}-|\nabla _{\Gamma
}u|^{2}\right) dS  \notag \\
& +\int_{\Gamma }\left( L_{2}(v)-L_{2}(u)\right) \frac{dS}{b}  \notag \\
& \geq \int_{\Omega }\nabla u\cdot \nabla (v-u)dx+\int_{\Omega }\alpha
_{1}(u)(v-u)dx  \notag \\
& +\int_{\Gamma }cu(v-u)\frac{dS}{b}+q\int_{\Gamma }\nabla _{\Gamma }u\cdot
\nabla _{\Gamma }(v-u)dS+\int_{\Gamma }\alpha _{2}(u)(v-u)\frac{dS}{b}.
\notag
\end{align}%
Thus, from Definition \ref{def-weak-sol} and (\ref{subdiff5}), for all $%
(V-U)\in \mathbb{D}_{q}$, it follows that%
\begin{equation*}
J(V)-J(U)\geq \int_{\Omega }f(v-u)dx+\int_{\Gamma }g(v-u)\frac{dS}{b}.
\end{equation*}%
This inequality is also true for $V=U+W\in \mathbb{D}_{q}$, for some
arbitrary $W\in \mathbb{D}_{q}$. Indeed, let $W=(w,tr(w))^{T}\in \mathbb{D}%
_{q}$ be fixed, $w_{m}:=[w\wedge m]\vee (-m)$ and set $%
W_{m}:=(w_{m},tr(w_{m}))^{T}$. Let $W_{m,n}=(w_{m,n},tr(w_{m,n}))^{T}$ be a
sequence in $\mathbb{D}_{q}$ such that $-m\leq w_{m,n}\leq m$, $%
w_{m,n}\rightarrow w_{m}$ in $H^{1}(\Omega )$ and $tr(w_{m,n})\rightarrow
tr(w_{m})$ in $H^{1}(\Gamma ),$ if $q>0,$ as $n\rightarrow \infty $. Then,%
\begin{align}
J(W_{m}+U)-J(U)& =\lim_{n\rightarrow \infty }J(W_{m,n}+U)-J(U)
\label{subdiff6} \\
& \geq \lim_{n\rightarrow \infty }\left( \int_{\Omega
}fw_{m,n}dx+\int_{\Gamma }gw_{m,n}\frac{dS}{b}\right)  \notag \\
& \geq \int_{\Omega }fw_{m}\;dx+\int_{\Gamma }gw_{m}\frac{dS}{b}.  \notag
\end{align}%
Passing to the limit as $m\rightarrow \infty $ in (\ref{subdiff6})\ in a
standard way and using the fact $W\in \mathbb{D}_{q}$ is arbitrary, we
immediately get%
\begin{equation}
J(W+U)-J(U)\geq \int_{\Omega }fwdx+\int_{\Gamma }gw\frac{dS}{b}.
\label{subdiff7}
\end{equation}%
Since $\mathbb{D}_{q}$ is a vector space, we also obtain the corresponding
inequality (\ref{subdiff7}) when replacing $W+U$ by $V$. Hence, $F\in
\partial J(U)$ and this completes the proof of the claim.

We have shown that the (single-valued) subdifferential of the functional $J$
at $U$ is given by%
\begin{equation}
D(\partial J)=\left\{ (u,tr\left( u\right) )^{T}\in \mathbb{D}_{q}:-\Delta
u+\alpha _{1}(u)\in L^{2}(\Omega ),\text{ }b(x)\frac{\partial u}{\partial n}%
-qb(x)\Delta _{\Gamma }u+\alpha _{2}(u)\in L^{2}(\Gamma )\right\}
\label{J-1}
\end{equation}%
and%
\begin{equation}
\partial J(U)=\left( -\Delta u+\alpha _{1}\left( u\right) ,b\left( x\right)
\frac{\partial u}{\partial n}+c\left( x\right) u-qb\left( x\right) \Delta
_{\Gamma }u+\alpha _{2}\left( u\right) \right) ^{T}.  \label{J-2}
\end{equation}




Since the functional $J$ is proper, convex and lower-semicontinuous, it
follows from Minty's theorem \cite{Mi} that the operator $B:=\partial J$ is
maximal monotone (or $-B$ is m-dissipative), for our choice of the function $%
j_{2}\left( x,u\right) $ in (\ref{4.2}). Thus, the first result of this
section is the following.

\begin{theorem}
\label{t4.1} The operator $B$ is the subdifferential of a proper, convex,
lower semicontinuous function on $\mathbb{X}_{2}$.
\end{theorem}

Theorem \ref{t4.1} applies to both $A$, the negative Wentzell Laplacian (by
taking both $\alpha _{1}$ and $\alpha _{2}$ to be zero) and to the operator
governing (\ref{1.7}) on $\mathbb{X}_{2}$. We remark that the above
construction leads easily to a proof that the Wentzell Laplacian has a
compact resolvent. Of course, this follows easily from the results quoted in
Section $3$, but the compactness does not require $C^{\infty }$-regularity.

Next, let $A_{2}U=\left( \alpha _{1}\left( u\right) ,\alpha _{2}\left(
v\right) \right) ^{T},$ for every $U\in D\left( A_{2}\right) ,$ where%
\begin{equation}
D\left( A_{2}\right) =\left\{ \left( u,v)\right) ^{T}\in \mathbb{X}%
_{2}:\left( \alpha _{1}\left( u\right) ,\alpha _{2}\left( v\right) \right)
^{T}\in \mathbb{X}_{2}\right\} .  \label{4.7}
\end{equation}%
Define the functional $J_{2}:\;\mathbb{X}_{2}\rightarrow \lbrack 0,+\infty ]$
by%
\begin{equation*}
J_{2}(U)=%
\begin{cases}
\int_{\Omega }L_{1}(u)dx+\int_{\Gamma }L_{2}(v)\frac{dS}{b(x)},\;\; & \text{%
if }(u,v)^{T}\in D(J_{2}) \\
+\infty & \text{if }(u,v)^{T}\in \mathbb{X}_{2}\backslash D(J_{2}),%
\end{cases}%
\end{equation*}%
with effective domain%
\begin{equation*}
D(J_{2}):=\{(u,v)^{T}\in \mathbb{X}_{2}:\int_{\Omega }\Lambda
_{1}(u)dx+\int_{\Gamma }\Lambda _{2}(v)\frac{dS}{b(x)}<\infty \}.
\end{equation*}%
It is easy to see that, under the assumption \textbf{(H1)} on $\alpha _{i}$,
the functional $J_{2}$ is proper, convex and lower-semicontinuous on $%
\mathbb{X}_{2}$. We have the following.

\begin{lemma}
Let $\alpha _{i}:\;\mathbb{R}\rightarrow \mathbb{R}$, $i=1,2,$ satisfy
\textbf{(H1)}-\textbf{(H2)}. Then the subdifferential $\partial J_{2}$ and
the operator $A_{2}$ coincide, that is, $D(\partial J_{2})=D(A_{2})$ and,
for all $U:=(u,v)^{T}\in D(A_{2}),$ we have%
\begin{equation*}
\partial J_{2}\left( U\right) =A_{2}U=\left( \alpha _{1}\left( u\right)
,\alpha _{2}\left( v\right) \right) ^{T}.
\end{equation*}
\end{lemma}

\begin{proof}
Note that \textbf{(H1)} implies that $\partial J_{2}$ is a single valued
operator. Let $U=(u,v)^{T}\in D(J_{2})$ and $(f,g)^{T}=\partial J_{2}(U)$.
Then, by definition, $(f,g)^{T}\in \mathbb{X}_{2}$ and for every $%
V:=(u_{1},v_{1})^{T}\in D(J_{2}),$ we get%
\begin{equation}
\int_{\Omega }f(u_{1}-u)dx+\int_{\Gamma }g(v_{1}-v)\frac{dS}{b(x)}\leq
J_{2}(V)-J_{2}(U).  \label{comp-sub}
\end{equation}%
Next, let $W=(u,v)^{T}+t(u_{2},v_{2})^{T},$ with $(u_{2},v_{2})^{T}\in
D(J_{2})$ and $0<t\leq 1$. Since \textbf{(H2)} implies that $D(J_{2})$ is a
vector space, then $W\in D(J_{2})$. Now, replacing $V$ in \eqref{comp-sub}
with $W$, dividing by $t$ and taking the limit as $t\rightarrow 0^{+}$
(where we make use of the Lebesgue Dominated Convergence theorem once
again), we obtain%
\begin{equation}
\int_{\Omega }fu_{2}dx+\int_{\Gamma }gv_{2}\frac{dS}{b(x)}\leq \int_{\Omega
}\alpha _{1}(u)u_{2}dx+\int_{\partial \Omega }\alpha _{2}(v)v_{2}\,\frac{dS}{%
b(x)}.  \label{first}
\end{equation}%
Changing $(u_{2},v_{2})^{T}$ to $-(u_{2},v_{2})^{T}$ in (\ref{first}) gives
\begin{equation*}
\int_{\Omega }fu_{2}dx+\int_{\Gamma }gv_{2}\frac{dS}{b(x)}=\int_{\Omega
}\alpha _{1}(u)u_{2}dx+\int_{\partial \Omega }\alpha _{2}(v)v_{2}\,\frac{dS}{%
b(x)}.
\end{equation*}%
In particular, taking $v_{2}=0$, for every $u_{2}\in C_{0}^{\infty }(\Omega
) $, we have%
\begin{equation*}
\int_{\Omega }fu_{2}dx=\int_{\Omega }\alpha _{1}(u)u_{2}dx,
\end{equation*}%
and this shows that $\alpha _{1}(u)=f$. Similarly, one obtains $\alpha
_{2}(v)=g$. We have shown that $U:=(u,v)^{T}\in D(A_{2})$ and $\partial
J_{2}(U)=(\alpha _{1}(u),\alpha _{2}(v))^{T}$.

Conversely, let $U=(u,v)^{T}\in D(A_{2})$ and set $(f,g)^{T}:=A_{2}U=(\alpha
_{1}(u),\alpha _{2}(v))^{T}$. Observe preliminarily that, owing to \textbf{%
(H2)}, there exist constants $t_{i}>0$ and $k_{i}\in (0,1]$ such that
\begin{equation}
k_{i}t\alpha _{i}(t)\leq \Lambda _{i}(t)\leq t\alpha _{i}(t)\text{, for all }%
|t|\geq t_{i},\text{ }i=1,2.  \label{delta-2}
\end{equation}%
Since $(\alpha _{1}(u),\alpha _{2}(v))^{T}\in \mathbb{X}_{2},$ from (\ref%
{delta-2}), it follows that%
\begin{eqnarray*}
\int_{\Omega }\Lambda _{1}(u)dx &=&\int_{\{x\in \Omega
:\;|u(x)|<t_{1}\}}\Lambda _{1}(u)dx+\int_{\{x\in \Omega :\;|u(x)|\geq
t_{1}\}}\Lambda _{1}(u)dx \\
&\leq &|\Omega |(\Lambda _{1}(t_{1})+\Lambda _{1}(-t_{1}))+\int_{\Omega
}u\alpha _{1}(u)dx<\infty ,
\end{eqnarray*}%
where a similar inequality holds for $\Lambda _{2}$. Hence%
\begin{equation*}
\int_{\Omega }\Lambda _{1}(u)dx+\int_{\partial \Omega }\Lambda _{2}(v)\frac{%
dS}{b(x)}<\infty
\end{equation*}%
and this shows that $(u,v)^{T}\in D(J_{2})$. Let $V=(u_{1},v_{1})^{T}\in
D(J_{2})$. Note that by \eqref{eq4-18-1} and \eqref{eq4-19-2}, we have once
more that%
\begin{equation}
\alpha _{1}(u)(u_{1}-u)\leq L_{1}(u_{1})-L_{1}(u)  \label{eq4-18}
\end{equation}%
and
\begin{equation}
\alpha _{2}(v)(v_{1}-v)\leq L_{2}(v_{1})-L_{2}(v).  \label{eq4-19}
\end{equation}%
Therefore, on account of \eqref{eq4-18}-\eqref{eq4-19}, it follows that
\begin{align*}
\int_{\Omega }f(u_{1}-u)dx+\int_{\Gamma }g(v_{1}-v)\frac{dS}{b(x)}&
=\int_{\Omega }\alpha _{1}(u)(u_{1}-u)dx+\int_{\partial \Omega }\alpha
_{2}(v)(v_{1}-v)\frac{dS}{b(x)} \\
& \leq J_{2}(V)-J_{2}(U).
\end{align*}%
By definition, we have shown that $(\alpha _{1}(u),\alpha
_{2}(v))^{T}=\partial J_{2}(U)$. Hence, $U\in D(\partial J_{2})$ and $%
A_{2}U=\partial J_{2}(U).$ This completes the proof.
\end{proof}

We will need the following results from semigroup theory and convex analysis.

\begin{definition}[\protect\cite{BH}]
\label{D1} Let $\mathcal{H}$ be a real Hilbert space. Two subsets $K_{1}$
and $K_{2}$ are almost equal, written as $K_{1}\simeq K_{2},$ if $K_{1}$ and
$K_{2}$ have the same closure and the same interior, that is, $\overline{%
K_{1}}=\overline{K_{2}}$ and $int\left( K_{1}\right) =int\left( K_{2}\right)
.$
\end{definition}

The following result is contained in \cite[pp.173--174]{BH}.

\begin{theorem}
\label{T2} Let $A$ and $B$ be subdifferentials of proper convex lower
semicontinuous functionals $\varphi _{1}$ and $\varphi _{2}$, respectively,
on a real Hilbert space $\mathcal{H}$ with $D(\varphi _{1})\cap D(\varphi
_{2})\neq \emptyset $. Let $C$ be the subdifferential of the proper, convex
lower semicontinuous functional $\varphi _{1}+\varphi _{2}$, that is, $%
C=\partial (\varphi _{1}+\varphi _{2})$. Then
\begin{equation}
\mathcal{R}(A)+\mathcal{R}(B)\subset \overline{\mathcal{R}(C)}\;\;\;%
\mbox{
and }\;\;\;\mbox{Int}\left( \mathcal{R}(A)+\mathcal{R}(B)\right) \subset %
\mbox{Int}\left( \mathcal{R}(C)\right) .  \label{resl-BH}
\end{equation}%
In particular, if the operator $A+B$ is maximal monotone, then
\begin{equation}
\mathcal{R}\left( A+B\right) \simeq \mathcal{R}\left( A\right) +\mathcal{R}%
\left( B\right)  \label{resl-BH-2}
\end{equation}%
and this is the case, if $\partial (\varphi _{1}+\varphi _{2})=\partial
\varphi _{1}+\partial \varphi _{2}$.
\end{theorem}

Here by $\mathcal{R}\left( A\right) +\mathcal{R}\left( B\right) $ we mean%
\begin{eqnarray*}
&&\cup \left\{ Af+Bg:f\in D\left( A\right) ,g\in D\left( B\right) \right\}
\newline
\\
&=&\cup \left\{ h+k:\left( f,h\right) \in A,\left( g,k\right) \in B\text{
for some }f,g\in \mathcal{H}\right\} .
\end{eqnarray*}%
We use the union symbol since $A$ and $B$ may be multi-valued. However, in
our applications, $A$ and $B$ will be single valued.

Let us recall that we want to solve the following problem:%
\begin{equation}
\left\{
\begin{array}{c}
-\Delta u+\alpha _{1}\left( u\right) =f_{1}\left( x\right) \text{ in }\Omega
, \\
b\left( x\right) \frac{\partial u}{\partial n}+c\left( x\right) u-qb\left(
x\right) \Delta _{\Gamma }u+\alpha _{2}\left( u\right) =f_{2}\left( x\right)
\text{ on }\Gamma .%
\end{array}%
\right.  \label{4.5}
\end{equation}

In order to solve (\ref{4.5}), recall that $A$ is the linear operator,
defined in Section 2 (see (\ref{2.12})). More precisely, $A$ has the
following operator representation:%
\begin{equation}
A=\left(
\begin{array}{cc}
-\Delta  & \qquad 0 \\
b\frac{\partial }{\partial n} & \qquad cI-qb\Delta _{\Gamma }%
\end{array}%
\right) .  \label{4.6}
\end{equation}%
Denote the null space of $A$ by $\mathcal{N}\left( A\right) .$ Then $%
U=(u,tr(u))^{T}\in \mathcal{N}\left( A\right) $ if and only if (by
definition) $u$ is a weak solution of%
\begin{equation}
\left\{
\begin{array}{c}
-\Delta u=0\text{ in }\Omega , \\
b\left( x\right) \frac{\partial u}{\partial n}+c\left( x\right) u-qb\left(
x\right) \Delta _{\Gamma }u=0\text{ on }\Gamma ,%
\end{array}%
\right.   \label{4.8}
\end{equation}%
that is, $u\in H^{1}(\Omega )$ with $tr(u)\in H^{1}(\Gamma )$ if $q>0$ and
\begin{equation}
\int_{\Omega }\nabla u\cdot \nabla vdx+\int_{\Gamma }cuv\frac{dS}{b}%
+q\int_{\Gamma }\nabla _{\Gamma }u\cdot \nabla _{\Gamma }vdS=0,
\label{weak-2}
\end{equation}%
for all $v\in H^{1}(\Omega )$ with $tr(v)\in H^{1}(\Gamma )$ if $q>0$. In
this case it is easy to see that $u$ is a weak solution of \eqref{4.8} if
and only if $u\in H^{1}(\Omega )$ with $tr(u)\in H^{1}(\Gamma ),$ if $q>0,$
and \eqref{weak-2} holds for all $v\in H^{1}(\Omega )$ with $tr(v)\in
H^{1}(\Gamma ),$ if $q>0$. Hence, it is clear that the null space of $A$ is $%
\mathcal{N}\left( A\right) =\mathbb{R}\mathbf{1}=\left\{ C\mathbf{1:}\,C\in
\mathbb{R}\right\} $ if $c\equiv 0$ in (\ref{4.8}), that is, $\mathcal{N}%
\left( A\right) $ consists of all the real constant functions on $\overline{%
\Omega }.$ We shall discuss this case first.

From now on, let $A_{1}$ be the linear operator $A$ corresponding to the
case of $c\equiv 0$. Moreover, let $A_{3}$ be the subdifferential $\partial
J $ of the functional $J,$ defined in \eqref{4.1}-(\ref{4.2}), that is, $%
A_{3}:=\partial J=\partial \left( J_{1}+J_{2}\right) $ (see \eqref{J-1}-%
\eqref{4.7}). It follows, from the assumptions on the functions $\alpha
_{1}, $ $\alpha _{2}$ and the results of Section 2, that $A_{i}=\partial
J_{i}$, for each $i=1,2,3$, where each $J_{i}$ is a proper, convex and lower
semicontinuous functional on $\mathbb{X}_{2}$.

Let us recall the Fredholm alternative, which says that for any selfadjoint
operator $B$ with compact resolvent and $0\notin \rho \left( B\right) ,$ we
have that the range $\mathcal{R}\left( B\right) =\overline{\mathcal{R}\left(
B\right) }=\mathcal{N}\left( B\right) ^{\perp }.$ This is the case with our
operator $A_{1},$ that is, we have%
\begin{equation}
\mathcal{R}\left( A_{1}\right) =\mathcal{N}\left( A_{1}\right) ^{\perp }=%
\mathbf{1}^{\perp }=\left\{ F\in \mathbb{X}_{2}:\int\limits_{\overline{%
\Omega }}Fd\mu =0\right\} ,  \label{4.9}
\end{equation}%
where the measure $\mu $ is defined by $d\mu =dx|_{\Omega }\oplus \frac{dS}{b%
}|_{\Gamma }$ on $\overline{\Omega }$. Let us now define $\lambda _{1},$ $%
\lambda _{2}\in \mathbb{R}_{+}$ by%
\begin{equation}
\lambda _{1}=\int\limits_{\Omega }dx,\text{ }\lambda
_{2}=\int\limits_{\Gamma }\frac{dS}{b}  \label{4.10}
\end{equation}%
so that $\mu \left( \overline{\Omega }\right) =\lambda _{1}+\lambda _{2}.$
We also define the average of $F$ with respect to the measure $\mu $, as
follows:%
\begin{equation}
ave_{\mu }\left( F\right) :=\frac{1}{\mu \left( \overline{\Omega }\right) }%
\int\limits_{\overline{\Omega }}Fd\mu =\frac{1}{\mu \left( \overline{\Omega }%
\right) }\left( \int\limits_{\Omega }f_{1}dx+\int\limits_{\Gamma }f_{2}\frac{%
dS}{b}\right) ,  \label{4.11}
\end{equation}%
for every $F=\left( f_{1},f_{2}\right) ^{T}\in \mathbb{X}_{2}.$

We now restate Theorem \ref{main}.

\begin{theorem}
\label{t4.4} Let $\alpha _{i}:\;\mathbb{R}\rightarrow \mathbb{R}$, $i=1,2,$
satisfy \textbf{(H1)}. Let $c\equiv 0$ in (\ref{4.5}) and let $F=\left(
f_{1},f_{2}\right) ^{T}\in \mathbb{X}_{2}$. A necessary condition for the
existence of a weak solution of (\ref{4.5}) is%
\begin{equation}
ave_{\mu }\left( F\right) \in \frac{\lambda _{1}\mathcal{R}\left( \alpha
_{1}\right) +\lambda _{2}\mathcal{R}\left( \alpha _{2}\right) }{\lambda
_{1}+\lambda _{2}},  \label{4.12}
\end{equation}%
while a sufficient condition is that $\alpha _{i}$ satisfies \textbf{(H2)}
and%
\begin{equation}
ave_{\mu }\left( F\right) \in int\left( \frac{\lambda _{1}\mathcal{R}\left(
\alpha _{1}\right) +\lambda _{2}\mathcal{R}\left( \alpha _{2}\right) }{%
\lambda _{1}+\lambda _{2}}\right) .  \label{4.13}
\end{equation}%
Assuming \textbf{(H2)}, the condition (\ref{4.12}) is both necessary and
sufficient when $\lambda _{1}\mathcal{R}\left( \alpha _{1}\right) +\lambda
_{2}\mathcal{R}\left( \alpha _{2}\right) $ is open, which holds if at least
one of $\mathcal{R}\left( \alpha _{1}\right) ,$ $\mathcal{R}\left( \alpha
_{2}\right) $ is open.
\end{theorem}

\begin{proof}
Let $\alpha _{i}:\;\mathbb{R}\rightarrow \mathbb{R}$, $i=1,2,$ satisfy
\textbf{(H1)}. Let $F=\left( f_{1},f_{2}\right) ^{T}\in \mathbb{X}_{2}$ be
given and let $u$ be a weak solution of (\ref{4.5}) with $c\equiv 0$. Then
(see Definition \ref{def-weak-sol}), $u\in H^{1}(\Omega )$, $\alpha
_{1}(u)\in L^{1}(\Omega ),$ $\alpha _{2}(tr(u))\in L^{1}(\Gamma )$, $%
tr(u)\in H^{1}(\Gamma )$ if $q>0$ and
\begin{align}
\int_{\Omega }f_{1}vdx+\int_{\Gamma }f_{2}v\frac{dS}{b}& =\int_{\Omega
}\nabla u\cdot \nabla vdx  \label{eq-weak-solu-2} \\
+\int_{\Omega }\alpha _{1}(u)vdx& +\int_{\Gamma }\alpha _{2}(u)v\frac{dS}{b}%
+q\int_{\Gamma }\nabla _{\Gamma }u\cdot \nabla _{\Gamma }vdS,  \notag
\end{align}%
for all $v\in H^{1}(\Omega )\cap C\left( \overline{\Omega }\right) ,$ if $%
q=0,$ and all $v\in H^{1}(\Omega )\cap C\left( \overline{\Omega }\right) $
with $tr(v)\in H^{1}(\Gamma ),$ if $q>0$. Taking $v=1$ in (\ref%
{eq-weak-solu-2}), we obtain%
\begin{align*}
\int_{\bar{\Omega}}Fd\mu =\int\limits_{\Omega }f_{1}dx+\int\limits_{\Gamma
}f_{2}\frac{dS}{b}& =\int\limits_{\Omega }\alpha _{1}\left( u\right)
dx+\int\limits_{\Gamma }\alpha _{2}\left( u\right) \frac{dS}{b} \\
& \in \left( \lambda _{1}\mathcal{R}\left( \alpha _{1}\right) +\lambda _{2}%
\mathcal{R}\left( \alpha _{2}\right) \right) ,
\end{align*}%
and so (\ref{4.12}) holds. This proves the necessity.\newline

For the sufficiency, let (\ref{4.13}) hold and assume that $\alpha _{i}$
satisfies \textbf{(H2)}. To show that \eqref{4.5}, with $c\equiv 0,$ has a
weak solution $u$, it is enough to prove that $F:=(f_{1},f_{2})\in \mathcal{R%
}(A_{3})$. To this end, we will make use of \eqref{resl-BH} from Theorem \ref%
{T2} to show that $F\in int(\mathcal{R}(A_{1})+\mathcal{R}(A_{2}))\subset
\mathcal{R}(A_{3})$. We know that $-A_{1},$ $-A_{2}$ and $-A_{3}$ are
m-dissipative on $\mathbb{X}_{2}$ and $A_{i}=\partial J_{i},$ for every $%
i=1,2,3,$ where each $J_{i},$ $i=2,3,$ is a proper, convex and lower
semicontinuous functional on $\mathbb{X}_{2}.$ Here, $J_{3}=J_{1}+J_{2}$ has
the effective domain $D(J_{3})=D(J_{1})\cap D(J_{2})\neq \emptyset $.

Let $c_{1},$ $c_{2}\in \mathbb{R}$, $C=\left( c_{1},c_{2}\right) ^{T}\in
\mathbb{X}_{2}$ and let $\mathcal{C}$ be the family of such vectors $C$ in $%
\mathbb{X}_{2}$. Let
\begin{equation*}
Q:=\left\{ C\in \mathcal{C}:c_{i}\in \mathcal{R}\left( \alpha _{i}\right) ,%
\text{ }i=1,2\right\} .
\end{equation*}%
Clearly $Q\subset \mathcal{R}\left( A_{2}\right) ,$ since $c_{i}=\alpha
_{i}\left( d_{i}\right) $ for some constant function $d_{i}$ on $\Omega $
(if $i=1$) or on $\Gamma $ (if $i=2$). Now let (\ref{4.13}) hold for $F\in
\mathbb{X}_{2}$. We must show $F\in \mathcal{R}\left( A_{3}\right) .$ By %
\eqref{4.13} we may choose $C=\left( c_{1},c_{2}\right) ^{T}\in Q$ such that%
\begin{equation*}
ave_{\mu }\left( F\right) =\frac{\lambda _{1}c_{1}+\lambda _{2}c_{2}}{%
\lambda _{1}+\lambda _{2}}\in int\left( \frac{\lambda _{1}\mathcal{R}\left(
\alpha _{1}\right) +\lambda _{2}\mathcal{R}\left( \alpha _{2}\right) }{%
\lambda _{1}+\lambda _{2}}\right) ,
\end{equation*}%
where $\lambda _{1},\lambda _{2}$ are given by (\ref{4.10}). Then, for $F\in
\mathbb{X}_{2},$ we have%
\begin{equation*}
F=\left[ F-C\right] +C.
\end{equation*}%
First, $F-C\in \mathcal{R}\left( A_{1}\right) =\mathcal{N}\left(
A_{1}\right) ^{\perp }=\mathbf{1}^{\perp },$ since%
\begin{align*}
\int\limits_{\overline{\Omega }}\left( F-C\right) d\mu & =\int\limits_{%
\overline{\Omega }}Fd\mu -\left( \lambda _{1}c_{1}+\lambda _{2}c_{2}\right)
\\
& =\int\limits_{\overline{\Omega }}\left[ F-ave_{\mu }\left( F\right) \right]
d\mu =0.
\end{align*}%
Next, clearly $C\in \mathcal{R}\left( A_{2}\right) .$ Thus, it is readily
seen that $F\in \left( \mathcal{R}\left( A_{1}\right) +\mathcal{R}\left(
A_{2}\right) \right) $. Let now $\varepsilon >0$ be given. We want $%
\varepsilon >0$ to be small enough, in particular, suppose%
\begin{equation*}
0<\varepsilon <\frac{1}{2}dist\left( \frac{\lambda _{1}c_{1}+\lambda
_{2}c_{2}}{\lambda _{1}+\lambda _{2}},\mathbf{K}\right) ,
\end{equation*}%
where $\mathbf{K}$ consists of the endpoints of the interval $\widetilde{I}%
=\left( \lambda _{1}\mathcal{R}\left( \alpha _{1}\right) +\lambda _{2}%
\mathcal{R}\left( \alpha _{2}\right) \right) /\left( \lambda _{1}+\lambda
_{2}\right) .$ Let $\widetilde{F}=(\widetilde{f_{1}},\widetilde{f_{2}}%
)^{T}\in \mathbb{X}_{2}$ satisfy $\left\Vert F-\widetilde{F}\right\Vert _{%
\mathbb{X}_{2}}<\varepsilon .$ We want to pick $\widetilde{C}=\left(
\widetilde{c}_{1},\widetilde{c}_{2}\right) ^{T}\in Q$ such that%
\begin{equation}
\left\Vert C-\widetilde{C}\right\Vert _{\mathbb{X}_{2}}<\varepsilon \text{
and }ave_{\mu }\left( \widetilde{F}\right) =\frac{\lambda _{1}\widetilde{c}%
_{1}+\lambda _{2}\widetilde{c}_{2}}{\lambda _{1}+\lambda _{2}}.  \label{est1}
\end{equation}%
To see how to do this, let $\mathcal{J}_{i}=\mathcal{R}\left( \alpha
_{i}\right) $ for $i=1,2.$ Then $c_{i}\in \mathcal{J}_{i}$ and%
\begin{equation}
ave_{\mu }\left( F\right) =\frac{\lambda _{1}c_{1}+\lambda _{2}c_{2}}{%
\lambda _{1}+\lambda _{2}}\in int\left( \frac{\lambda _{1}\mathcal{J}%
_{1}+\lambda _{2}J_{2}}{\lambda _{1}+\lambda _{2}}\right) .  \label{est2}
\end{equation}%
We may choose at least one of $\widetilde{c}_{1},$ $\widetilde{c}_{2},$ call
it $\tilde{c}_{k},$ to be less than $c_{k},$ because $c_{k}$ cannot be the
left hand end point of $\mathcal{J}_{k}$ for both $k=1,2,$ because of (\ref%
{4.13}). In a similar way, we may choose one of $\widetilde{c}_{1},$ $%
\widetilde{c}_{2},$ call it $\widetilde{c}_{l},$ to be larger than $c_{l}.$
Next,%
\begin{equation*}
\left\vert ave_{\mu }\left( F\right) -ave_{\mu }\left( \widetilde{F}\right)
\right\vert \leq \left\Vert F-\widetilde{F}\right\Vert _{\mathbb{X}%
_{2}}<\varepsilon ,
\end{equation*}%
by the Schwarz inequality. By this observation and (\ref{est1})-(\ref{est2}%
), we can find $\widetilde{C}=\left( \widetilde{c}_{1},\widetilde{c}%
_{2}\right) \in Q$ such that (\ref{est1}) holds. Thus, $\left( \mathcal{R}%
\left( A_{1}\right) +\mathcal{R}\left( A_{2}\right) \right) $ contains an $%
\varepsilon $-ball in $\mathbb{X}_{2},$ centered at $F,$ for sufficiently
small $\varepsilon >0.$ Thus,%
\begin{equation*}
F\in int\left( \mathcal{R}\left( A_{1}\right) +\mathcal{R}\left(
A_{2}\right) \right) \subset int\left( \mathcal{R}\left( A_{3}\right)
\right) \subset \mathcal{R}\left( A_{3}\right) ,
\end{equation*}%
by \eqref{resl-BH}. Consequently,\ problem (\ref{4.5}) is (weakly) solvable
in the sense of Definition \ref{def-weak-sol}, for any $f_{1}\in L^{2}\left(
\Omega \right) ,$ $f_{2}\in L^{2}\left( \Gamma \right) ,$ if (\ref{4.13})
holds. This completes the proof.
\end{proof}

We will now give some examples as applications of Theorem \ref{t4.4}.

\begin{example}
\label{e4.5}Let $\alpha _{1}\left( s\right) $ or $\alpha _{2}\left( s\right)
$ be equal to $\alpha \left( s\right) =r\left\vert s\right\vert ^{p-1}s,$
where $r,$ $p>0$. Then, it is clear that $\alpha $ satisfies \textbf{(H1)}
and that $L(s)=\Lambda (s)=\frac{r}{p+1}|s|^{p+1}$ also satisfies \textbf{%
(H2)}. Note that $\mathcal{R}\left( \alpha \right) =\mathbb{R}$. Then, it
follows that problem (\ref{4.5}) with $c\equiv 0$ is solvable for any $%
f_{1}\in L^{2}\left( \Omega \right) ,$ $f_{2}\in L^{2}\left( \Gamma \right) $%
.
\end{example}

\begin{example}
\label{e4.6}Consider the case when $c=q=\alpha _{2}\equiv 0$ in (\ref{4.5}),
that is, consider the following boundary value problem:%
\begin{equation*}
\left\{
\begin{array}{c}
-\Delta u+\alpha _{1}\left( u\right) =f_{1}\left( x\right) \text{ in }\Omega
, \\
b\left( x\right) \frac{\partial u}{\partial n}=f_{2}\left( x\right) \text{
on }\Gamma .%
\end{array}%
\right.
\end{equation*}%
Then, by Theorem \ref{t4.4}, this problem has a weak solution if%
\begin{equation*}
\int\limits_{\Omega }f_{1}dx+\int\limits_{\Gamma }f_{2}\frac{dS}{b}\in
\lambda _{1}int\left( \mathcal{R}\left( \alpha _{1}\right) \right) ,
\end{equation*}%
which yields the classical Landesman-Lazer result (see (\ref{1.4})) for $%
f_{2}\equiv 0$.
\end{example}

\begin{example}
\label{e4.7}Let us now consider the case when $\alpha _{1}\equiv 0$ and $%
\alpha _{2}\equiv \alpha ,$ where $\alpha $ is a continuous, monotone
nondecreasing function on $\mathbb{R}$ such that $\alpha \left( 0\right) =0$%
. The problem%
\begin{equation}
\left\{
\begin{array}{c}
-\Delta u=f_{1}\left( x\right) \text{ in }\Omega , \\
b\left( x\right) \frac{\partial u}{\partial n}-qb\left( x\right) \Delta
_{\Gamma }u+\alpha \left( u\right) =f_{2}\left( x\right) \text{ on }\Gamma ,%
\end{array}%
\right.  \label{e}
\end{equation}%
has a weak solution if%
\begin{equation}
\int\limits_{\Omega }f_{1}dx+\int\limits_{\Gamma }f_{2}\frac{dS}{b}\in
\lambda _{2}int\left( \mathcal{R}\left( \alpha \right) \right) .  \label{ee}
\end{equation}%
For example, if we choose $\alpha \left( s\right) =\arctan \left( s\right) $
in (\ref{e}), (\ref{ee}) becomes the necessary and sufficient condition%
\begin{equation}
\left\vert \frac{1}{\lambda _{2}}\left( \int\limits_{\Omega
}f_{1}dx+\int\limits_{\Gamma }f_{2}\frac{dS}{b}\right) \right\vert <\frac{%
\pi }{2}.
\end{equation}%
Note that $\alpha (s)=\arctan (s)$ satisfies \textbf{(H1)} and that $%
L_{2}(s)=\Lambda _{2}(s)=s\arctan (s)-\ln \sqrt{1+s^{2}}$ satisfies \textbf{%
(H2)}.
\end{example}

Let us now turn to the case when $c>0$ on a set of positive $dS$-measure
(that is, $c\left( x\right) $ is not identically zero on the boundary $%
\Gamma $) and consider $A_{1}^{1}$ to be the linear operator $A$ of (\ref%
{2.12}) corresponding to this case. Since $A_{1}^{1}=\left( A_{1}^{1}\right)
^{\ast }\geq 0$ and $A_{1}^{1}$ has compact resolvent, it has a ground state
$Z=\left( z_{\mid \Omega },z_{\mid \Gamma }\right) ^{T}.$ That is, $\lambda
=\min \sigma \left( A_{1}^{1}\right) $ is a simple eigenvalue, $\lambda >0,$
and $\mathcal{N}\left( A_{1}^{1}-\lambda \right) =\left\{ CZ:C\in \mathbf{R}%
\right\} $ for some positive function $Z$ on $\overline{\Omega }$.

Before proceeding further, we find the ground state of $A_{1}^{1}$ in a
simple one-dimensional example. Let $\Omega =\left( 0,1\right) ,$ $\Gamma
=\left\{ 0,1\right\} $, $b_{0}=b_{1}=1,$ $q=0$ and $c_{0},c_{1}$ will be
specified in the sequel. Here $b_{j}=b\left( j\right) $ and $c_{j}=c\left(
j\right) .$ We will choose $c_{j}$, $j=0,1$ so that the smallest eigenvalue
of $A_{1}^{1}$ is $\lambda =1.$ The required positive solution of $%
z^{^{\prime \prime }}+z=0$ has the form $z\left( x\right) =\cos \left(
x-\delta \right) $ (times a constant, which we take to be $1$). We need to
choose $\delta $ so that $z>0$ in $\left[ 0,1\right] $ and choose $c_{0},$ $%
c_{1}$ such that $z$ satisfies the correct boundary conditions. The boundary
conditions are%
\begin{equation}
-z\left( j\right) +\left( -1\right) ^{j+1}z^{^{\prime }}\left( j\right)
+c_{j}z\left( j\right) =0,  \label{BC1}
\end{equation}%
for $j=0,1,$ since $\partial /\partial n=\left( -1\right) ^{j+1}\partial
/\partial x$ and $z^{^{\prime \prime }}\left( j\right) =-z\left( j\right) $
at $x=j\in \left\{ 0,1\right\} $. Since $z\left( 0\right) =\cos \left(
\delta \right) ,$ $z^{^{\prime }}\left( 0\right) =\sin \left( \delta \right)
,$ $z\left( 1\right) =\cos \left( 1-\delta \right) $ and $z^{^{\prime
}}\left( 1\right) =\sin \left( \delta -1\right) .$ Then (\ref{BC1}) implies%
\begin{equation}
c_{0}=1+\tan \left( \delta \right) ,\text{ }c_{1}=1+\tan \left( 1-\delta
\right) .  \label{cc}
\end{equation}%
For $\delta \in \left( 0.4,0.6\right) ,$ $c_{0}$ and $c_{1}$ are both
positive. Next, for $x\in \left[ 0,1\right] ,$ we have $\left( x-\delta
\right) \in \left( -1,1\right) \subset \left( -\frac{\pi }{2},\frac{\pi }{2}%
\right) ,$ whence $z$ is positive on $\left[ 0,1\right] .$ Moreover, for $%
x\in \left[ 0,1\right] ,$ $x-\delta \in \left[ -\delta ,1-\delta \right] ,$
then choosing $\delta =1/2,$ we have $\left\vert x-1/2\right\vert \leq 1/2$,
$\cos \left( x-1/2\right) \in \left[ \cos \frac{1}{2},1\right] $ and $%
c_{0}=c_{1}=1+\tan \left( 1/2\right) .$

Finally, we can use the above results to prove our first result for a
similar elliptic problem to (\ref{4.5}) in this new case. As an application
of \eqref{resl-BH} (see Theorem \ref{T2}), we obtain the following.

\begin{theorem}
\label{t4.8} Let $c$ be a nonnegative function which is positive on $\Gamma
_{1}\subset \Gamma ,$ where $\int\limits_{\Gamma _{1}}dS>0$. Let $q=0$ and
let $\alpha $ be a continuous, monotone nondecreasing function on $\mathbb{R}
$ such that $\alpha \left( 0\right) =0$, $\alpha \left( \pm \infty \right) =%
\underset{s\rightarrow \pm \infty }{\lim }\alpha \left( s\right) $. Let $%
F=\left( f_{1},f_{2}\right) ^{T}\in \mathbb{X}_{2}.$ Also, suppose that $%
\lambda >0$ is the smallest eigenvalue of $A_{1}^{1}$ and let $Z$ be a
positive member of the one-dimensional eigenspace of $A_{4}:=A_{1}^{1}-%
\lambda I.$ Here we view $Z\in \mathbb{X}_{2}$ as $Z=(z_{1},z_{2})^{T}:%
\overline{\Omega }\rightarrow \mathbb{R}$, and $Z$ corresponds to a $%
z_{1}\in C(\overline{\Omega })$, with $z_{2}=z_{1}|_{\Gamma }$ and $z_{1}$
is a positive function on $\overline{\Omega }.$ A necessary condition for
the existence of a weak solution for%
\begin{equation}
\left\{
\begin{array}{c}
-\Delta u-\lambda u+\alpha \left( u\right) =f_{1}\text{ in }\Omega , \\
\Delta u+b\left( x\right) \frac{\partial u}{\partial n}+\left( c\left(
x\right) +\lambda \right) u=f_{2}\text{ on }\Gamma%
\end{array}%
\right.  \label{P2}
\end{equation}%
is%
\begin{equation}
\alpha \left( -\infty \right) \left\langle Z,\boldsymbol{1}\right\rangle _{%
\mathbb{X}_{2}}\leq \left\langle F,Z\right\rangle _{\mathbb{X}_{2}}\leq
\alpha \left( +\infty \right) \left\langle Z,\mathbf{1}\right\rangle _{%
\mathbb{X}_{2}},  \label{Nec}
\end{equation}%
while a sufficient condition is that $\alpha $ satisfies \textbf{(H2)} and%
\begin{equation}
\frac{\alpha \left( -\infty \right) }{\min Z}<\left\langle F,Z\right\rangle
_{\mathbb{X}_{2}}<\frac{\alpha \left( +\infty \right) }{\max Z}.  \label{Suf}
\end{equation}
\end{theorem}

\begin{proof}
For the necessity part, multiply the first equation of (\ref{P2}), the
second equation of (\ref{P2}) by $z$ and integrate by parts; here $Z=\left(
z|_{\Omega },z|_{\Gamma }\right) ^{T}.$ Using the divergence theorem and the
fact that $\mathcal{N}\left( A_{1}^{1}-\lambda \right) =span\left\{
Z\right\} ,$ we obtain%
\begin{equation*}
\int\limits_{\Omega }\alpha \left( u\right) zdx+\int\limits_{\Gamma }\alpha
\left( v\right) z_{\mid \Gamma }\frac{dS}{b}=\int\limits_{\Omega
}f_{1}zdx+\int\limits_{\Gamma }f_{2}z_{\mid \Gamma }\frac{dS}{b},
\end{equation*}%
where $U=\left( u,v\right) ^{T}$ with $v=tr\left( u\right) $ is the solution
of (\ref{P2}) with $F=\left( f_{1},f_{2}\right) ^{T}$. Since $Z>0$, this
equation becomes%
\begin{equation*}
\frac{\left\langle F,Z\right\rangle _{\mathbb{X}_{2}}}{\left\langle Z,%
\boldsymbol{1}\right\rangle _{\mathbb{X}_{2}}}=\frac{\left\langle \alpha
,Z\right\rangle _{\mathbb{X}_{2}}}{\left\langle Z,\boldsymbol{1}%
\right\rangle _{\mathbb{X}_{2}}}\in \left[ \alpha \left( -\infty \right)
,\alpha \left( +\infty \right) \right] ,
\end{equation*}%
and the necessary condition (\ref{Nec}) follows. If $\alpha \left( -\infty
\right) <\alpha \left( r\right) $ for all $r\in \mathbb{R}$, then the
endpoint $\alpha \left( -\infty \right) $ can be excluded. A similar remark
applies to $\alpha \left( +\infty \right) .$

The sufficiency proof is like that of Theorem \ref{t4.4}, but $Z$ is not a
constant. By the Fredholm alternative, we have%
\begin{equation}
\mathcal{R}\left( A_{4}\right) =\mathcal{N}\left( A_{4}\right) ^{\perp
}=\left\{ F\in \mathbb{X}_{2}:\left\langle F,Z\right\rangle _{\mathbb{X}%
_{2}}=0\right\} .  \label{FA}
\end{equation}

Let us also define the nonlinear operator $A_{5}U=\left( \alpha \left(
u\right) ,0\right) ^{T},$ for $\left( u,v\right) ^{T}\in D\left(
A_{5}\right) $ such that%
\begin{equation}
D\left( A_{5}\right) =\left\{ \left( u,v\right) ^{T}\in \mathbb{X}_{2}:u%
\text{ has a trace }tr\left( u\right) =v\text{ and }\left( \alpha \left(
u\right) ,0\right) ^{T}\in \mathbb{X}_{2}\right\} .
\end{equation}

Let us recall that, due to Theorem \ref{t4.1}, we know that $-A_{4},$ $%
-A_{5},$ are m-dissipative on $\mathbb{X}_{2}$ and $A_{i}=\partial J_{i},$
for every $i=4,5$ and each $J_{i}$ is a proper, convex and lower
semicontinuous functional on $\mathbb{X}_{2}$. Let $J_{6}:=J_{4}+J_{5}$ with
domain $D(J_{6}):=D(J_{4})\cap D(J_{5})\neq \emptyset .$ Then $J_{6}$ is a
proper, convex and lower semicontinuous functional on $\mathbb{X}_{2}$. Let $%
A_{6}:=\partial (J_{4}+J_{5})$. Then $-A_{6}$ is m-dissipative on $\mathbb{X}%
_{2}$. It follows, from \eqref{resl-BH}, that
\begin{equation}
\mathcal{R}\left( A_{4}\right) +\mathcal{R}\left( A_{5}\right) \subset
\overline{\mathcal{R}\left( A_{6}\right) }\;\text{and}\;int\left( \mathcal{R}%
\left( A_{4}\right) +\mathcal{R}\left( A_{5}\right) \right) \subset
int\left( \mathcal{R}\left( A_{6}\right) \right) .  \label{A4-A5}
\end{equation}%
Suppose now that $Z$ is a positive unit vector in $\mathcal{N}\left(
A_{4}\right) $ (recall that $A_{4}=A_{1}^{1}-\lambda I$)$,$ that is, $%
\lambda =\min \sigma \left( A_{4}\right) ,$ $A_{1}^{1}Z=\lambda Z,$ $%
\left\Vert Z\right\Vert _{\mathbb{X}_{2}}=1$ and $Z>0.$ For $F\in \mathbb{X}%
_{2},$ we have%
\begin{equation*}
F=\left[ F-\left\langle F,Z\right\rangle _{\mathbb{X}_{2}}Z\right]
+\left\langle F,Z\right\rangle _{\mathbb{X}_{2}}Z\in \mathcal{R}\left(
A_{4}\right) +\mathcal{R}\left( A_{5}\right) ,
\end{equation*}%
provided that%
\begin{equation*}
\alpha \left( -\infty \right) <\left\langle F,Z\right\rangle _{\mathbb{X}%
_{2}}Z<\alpha \left( +\infty \right)
\end{equation*}%
holds pointwise on $\overline{\Omega }$. But for, $\widetilde{F}=(\widetilde{%
f_{1}},\widetilde{f_{2}})^{T}\in \mathbb{X}_{2}$ and $\left\Vert F-%
\widetilde{F}\right\Vert _{\mathbb{X}_{2}}<\varepsilon ,$ we have again%
\begin{align}
\left\Vert \left\langle \widetilde{F},Z\right\rangle _{\mathbb{X}%
_{2}}Z-\left\langle F,Z\right\rangle _{\mathbb{X}_{2}}Z\right\Vert _{\mathbb{%
X}_{2}}& =\left\Vert \left\langle \widetilde{F}-F,Z\right\rangle _{\mathbb{X}%
_{2}}Z\right\Vert _{\mathbb{X}_{2}} \\
& \leq \left\Vert F-\widetilde{F}\right\Vert _{\mathbb{X}_{2}}<\varepsilon ,
\notag
\end{align}%
so then $\alpha \left( -\infty \right) <\left\langle \widetilde{F}%
,Z\right\rangle _{\mathbb{X}_{2}}Z<\alpha \left( +\infty \right) $ on $%
\overline{\Omega },$ for $\varepsilon >0$ small enough. It follows that%
\begin{equation*}
F\in int\left( \mathcal{R}\left( A_{4}\right) +\mathcal{R}\left(
A_{5}\right) \right) \subset int\left( \mathcal{R}\left( A_{6}\right)
\right) \subset \mathcal{R}\left( A_{6}\right) ,
\end{equation*}%
by \eqref{A4-A5}. This completes the proof of our theorem.
\end{proof}

\begin{remark}
When $\lambda =0$ and $Z\equiv \mathbf{1,}$ we have, using a different
normalization, $\left\Vert Z\right\Vert _{\mathbb{X}_{2}}^{2}=\mu \left(
\overline{\Omega }\right) =\lambda _{1}+\lambda _{2},$ $\min Z=\max Z=1;$ in
this case, it turns out that (\ref{Suf}) reduces to (\ref{4.13}).
\end{remark}

\begin{remark}
Of course the result in Theorem \ref{t4.8} is interesting only when
\begin{equation*}
\frac{\alpha \left( -\infty \right) }{\min Z}<\frac{\alpha \left( +\infty
\right) }{\max Z}.
\end{equation*}%
But this always holds unless $\alpha \equiv 0.$
\end{remark}

\begin{example}
In the context of Theorem \ref{t4.8}, let us now consider the one
dimensional problem:%
\begin{equation}
\left\{
\begin{array}{c}
-u^{^{\prime \prime }}+u+\alpha \left( u\right) =f_{1}\text{ in }\Omega
=\left( 0,1\right) , \\
-u\left( j\right) +\left( -1\right) ^{j+1}u^{^{\prime }}\left( j\right)
+c_{j}u\left( j\right) =f_{2}^{j}\text{, }j=0,1,%
\end{array}%
\right.  \label{P3}
\end{equation}%
where $c_{j}$ are given by (\ref{cc}) with $\delta =1/2.$ It follows from (%
\ref{Suf}) that, for (\ref{P3}) to have at least one solution, it suffices
to have%
\begin{equation}
\frac{\alpha \left( -\infty \right) }{\cos \left( 1/2\right) }%
<\int\limits_{0}^{1}f_{1}\left( x\right) \cos \left( x-1/2\right) dx+\left(
f_{2}^{0}+f_{2}^{1}\right) \cos \left( 1/2\right) <\alpha \left( +\infty
\right) .  \label{Suff}
\end{equation}

Moreover, choosing $\alpha \left( u\right) =r\left\vert u\right\vert
^{p-1}u, $ $r,p>0$ in the first equation of (\ref{P3}), then (\ref{Suff})
yields at least one solution to (\ref{P3}) for any $f_{1}\in L^{2}\left(
0,1\right) $ and $f_{2}^{j}\in \mathbb{R}$, $j=0,1.$
\end{example}

Finally, let us consider as an application of our main theorems, an example
for which $q>0,$ that is, $\Delta _{\Gamma }$ is present in the boundary
conditions for our nonlinear elliptic problems (\ref{P2}). For this purpose,
let $\Omega $ be the two dimensional box $\left( 0,1\right) ^{2}\subset
\mathbb{R}^{2}$, $b\left( x,y\right) \equiv 1,$ for all $\left( x,y\right)
\in \Gamma =\Gamma _{1}\cup \Gamma _{2}\cup \Gamma _{3}\cup \Gamma _{4}$, $%
q>0$ and $c_{i}\left( x,y\right) $ will be determined in the sequel. The
lines $\Gamma _{i}$ and $c_{i}$ will be defined below. We will choose $%
c_{i}\left( x,y\right) ,$ so that the smallest eigenvalue of $A_{1}^{1}$ is $%
\lambda =2$. The positive solution of $\Delta z+2z=0$ has the form $z\left(
x,y\right) =\cos \left( x-1/2\right) \cos \left( y-1/2\right) $ (times a
constant, which we take to be $1$). Note that $z\left( x,y\right) >0$ on $%
\overline{\Omega }=\left[ 0,1\right] ^{2}$. Thus, we need to choose positive
$c_{i}\left( x,y\right) $ for each $i=1,2,3,4$ such that $z\left( x,y\right)
$ satisfies the correct boundary conditions. The boundary conditions are%
\begin{equation}
\left\{
\begin{array}{c}
-2z-z_{y}+c_{1}\left( x,y\right) z-qz_{yy}=0\text{ for }\left( x,y\right)
\in \Gamma _{1}=\left\{ \left( x,0\right) :x\in \left[ 0,1\right] \right\} ,
\\
-2z+z_{x}+c_{2}\left( x,y\right) z-qz_{xx}=0\text{ for }\left( x,y\right)
\in \Gamma _{2}=\left\{ \left( 1,y\right) :y\in \left[ 0,1\right] \right\} ,
\\
-2z+z_{y}+c_{3}\left( x,y\right) z-qz_{yy}=0\text{ for }\left( x,y\right)
\in \Gamma _{3}=\left\{ \left( x,1\right) :x\in \left[ 0,1\right] \right\} ,
\\
-2z-z_{x}+c_{4}\left( x,y\right) z-qz_{xx}=0\text{ for }\left( x,y\right)
\in \Gamma _{4}=\left\{ \left( 0,y\right) :x\in \left[ 0,1\right] \right\} ,%
\end{array}%
\right.  \label{BCmany}
\end{equation}%
since $\partial /\partial n$ equals $\partial /\partial x$ and $\partial
/\partial y$ along the lines $\Gamma _{2}$ and $\Gamma _{3},$ respectively
and $\partial /\partial n$ equals $-\partial /\partial x$ and $-\partial
/\partial y$ along the lines $\Gamma _{4}$ and $\Gamma _{1},$ respectively.
Moreover, we note that $\Delta _{\Gamma }$ equals $\partial /\partial y^{2}$
along $\Gamma _{1}\cup \Gamma _{3}$ and $\partial /\partial x^{2}$ along $%
\Gamma _{2}\cup \Gamma _{4},$ respectively. Calculating in (\ref{BCmany}),
we obtain, for any $q\in \left( 0,q_{\pm }\right) ,$ $q_{\pm }=2\cos \left(
1/2\right) \pm \tan \left( 1/2\right) ,$ the functions%
\begin{equation}
\left\{
\begin{array}{c}
c_{1}\left( x,y\right) =q_{\pm }-q+d_{1}\left( y\right) , \\
c_{2}\left( x,y\right) =q_{\pm }-q+d_{2}\left( x\right) , \\
c_{3}\left( x,y\right) =q_{\pm }-q+d_{3}\left( y\right) , \\
c_{4}\left( x,y\right) =q_{\pm }-q+d_{4}\left( x\right) ,%
\end{array}%
\right.  \label{cccc}
\end{equation}%
where $d_{i}$ are nonnegative, continuous functions on $\left[ 0,1\right] $
such that $d_{1}\left( 0\right) =d_{4}\left( 0\right) =0$ and $d_{2}\left(
1\right) =d_{3}\left( 1\right) =0.$ Note that $c_{i}>0$ on $\Gamma _{i}$ for
each $i.$

\begin{example}
Let us now consider the boundary value problem in the open rectangle $\Omega
=\left( 0,1\right) ^{2}$:%
\begin{equation}
-\Delta u+2u+\alpha \left( u\right) =f_{1}\text{ in }\Omega ,  \label{P4}
\end{equation}%
endowed with the boundary conditions of (\ref{BCmany}), except that now the
zero values on the right sides of these equalities are replaced by the
functions $f_{2}^{1},$ $f_{2}^{2},$ $f_{2}^{3}$ and $f_{2}^{4},$
respectively. Let $c_{i}$ be the functions defined in (\ref{cccc}). It
follows from (\ref{Suf}) that for (\ref{P4}) to have at least one solution,
it suffices to have%
\begin{equation}
\frac{\alpha \left( -\infty \right) }{\cos ^{2}\left( 1/2\right) }<\mathcal{J%
}<\alpha \left( +\infty \right) ,  \label{last}
\end{equation}%
where%
\begin{equation*}
\mathcal{J}=\int\limits_{0}^{1}\int\limits_{0}^{1}f_{1}\left( x,y\right)
\cos \left( x-\frac{1}{2}\right) \cos \left( y-\frac{1}{2}\right)
dxdy+\sum\limits_{i=1}^{4}\int\limits_{\Gamma _{i}}f_{2}^{i}zdS_{i}
\end{equation*}%
and each $\int\limits_{\Gamma _{i}}dS_{i}$ denotes the path integral
corresponding to each line ${\Gamma _{i}}$. Moreover, choosing $\alpha
\left( u\right) =r\left\vert u\right\vert ^{p-1}u,$ $r,p>0$ in the (\ref{P4}%
), then (\ref{last}) yields at least one solution to (\ref{P4}), for any $%
f_{1}\in L^{2}\left( \Omega \right) $ and $f_{2}^{i}\in L^{2}\left( \Gamma
_{i}\right) $, $i=1,2,3,4.$
\end{example}

We conclude the paper by stating sufficient conditions for \eqref{resl-BH-2}
(see Theorem \ref{T2}) to hold. It is worth mentioning, however, that such
conditions are not necessary to prove Theorem \ref{t4.4}, but that the
results below have an interest on their own. We consider the following
growth conditions for a function $\alpha :\mathbb{R}\rightarrow \mathbb{R}$:

\textbf{(GC1) }$N=1$. No growth condition on $\alpha .$

\ \ \ \ \ \ \ \ \ \ $\ N=2.$ The function $\alpha $ is bounded by a power:%
\begin{equation}
\left\vert \alpha \left( s\right) \right\vert \leq C\left( 1+\left\vert
s\right\vert ^{r}\right) ,\text{ for all }s\in \mathbb{R}\text{,}  \label{p1}
\end{equation}%
where $C,$ $r$ are positive constants.

\ \ \ \ \ \ \ \ \ \thinspace $\ N=3$. (\ref{p1}) holds with $r=N/\left(
N-2\right) .$

\textbf{(GC2) }This is (GC1), modified by replacing $r=N/\left( N-2\right) $
by $r=\left( N-1\right) /\left( N-2\right) $ in the case $N\geq 3$ and $q>0,$
and replacing $r=N/\left( N-2\right) $ by%
\begin{equation*}
r=\left\{
\begin{array}{c}
\text{any number, if }N=3 \\
\frac{N-1}{N-3}\text{, if }N\geq 4.%
\end{array}%
\right.
\end{equation*}%
We start with the following.

\begin{proposition}
\label{theo-equ} Let $\alpha _{1},$ $\alpha _{2}:\mathbb{R}\rightarrow
\mathbb{R}$ satisfy \textbf{(H1)}. Assume that%
\begin{equation}
(\alpha _{1}(u),\alpha _{2}(u))^{T}\in \mathbb{X}_{2},\;\text{for all}\;u\in
H^{1}\left( \Omega \right) ,\;\text{if }q=0,  \label{L2-cond}
\end{equation}%
\begin{equation}
(\alpha _{1}(u),\alpha _{2}(u_{\mid \Gamma }))^{T}\in \mathbb{X}_{2},\;\text{%
for all }(u,tr\left( u\right) )^{T}\in H^{1}(\Omega )\times H^{1}(\Gamma ),\;%
\text{if}\;q>0.  \label{L2-cond-1}
\end{equation}%
Let $A_{1}$, $A_{2}$ and $A_{3}$ be as in the proof of Theorem \ref{t4.4}.
Then%
\begin{equation}
A_{1}+A_{2}=A_{3}\;\text{and}\;\mathcal{R}(A_{1})+\mathcal{R}(A_{2})\simeq
\mathcal{R}(A_{3}).  \label{conc}
\end{equation}
\end{proposition}


\begin{proof}
Let us first recall that, from Theorem \ref{t3.2}, $D(A_{1})$ equals either $%
H^{2}(\Omega )$ or $H_{\ast }^{2}(\Omega ),$ according to whether $q=0$ or $%
q>0$. Moreover,%
\begin{equation*}
A_{1}U=\left( -\Delta u,b(x)\partial _{n}u-qb(x)\Delta _{\Gamma }u\right)
^{T}.
\end{equation*}%
The operators $A_{2},$ $A_{3}$ are given in \eqref{4.7} and \eqref{J-1}-%
\eqref{J-2}, respectively. Since $A_{1}=\partial J_{1}$, $A_{2}=\partial
J_{2}$ and $A_{3}=\partial J_{3}:=\partial (J_{1}+J_{2})$ with $D(J_{1})\cap
D(J_{2})\neq \emptyset $, it follows that $A_{1}+A_{2}\subset A_{3}.$ Hence,
$D\left( A_{1}\right) \cap D\left( A_{2}\right) \subset D\left( A_{3}\right)
$. We claim that $A_{3}=A_{1}+A_{2}$. To show this we must prove%
\begin{equation*}
D\left( A_{3}\right) \subset D\left( A_{1}\right) \cap D\left( A_{2}\right) .
\end{equation*}%
Assume \eqref{L2-cond} and let $U=(u,u_{\mid \Gamma })^{T}\in D(A_{3})$.
Then $U\in \mathbb{D}_{0}$, and from (\ref{J-1}),%
\begin{equation*}
-\Delta u+\alpha _{1}(u)\in L^{2}(\Omega ),\text{ }\frac{\partial u}{%
\partial n}+\alpha _{2}(u)\in L^{2}(\Gamma )\text{, if }q=0.
\end{equation*}%
Therefore, $u\in H^{1}(\Omega )$, $\Delta u\in L^{2}(\Omega )$ and $\frac{%
\partial u}{\partial n}\in L^{2}(\Gamma )$. Since $\Omega $ is smooth,
elliptic regularity implies that $u\in H^{2}(\Omega )$. Hence, $U\in D\left(
A_{1}\right) \cap D\left( A_{2}\right) ,$ if $q=0$. If $q>0$, one also has
that $\frac{\partial u}{\partial n}-qb(x)\Delta _{\Gamma }u+\alpha
_{2}(u)\in L^{2}(\Gamma )$ and $tr\left( u\right) \in H^{1}(\Gamma )$. Since
$u\in H^{2}(\Omega )$, and $\alpha _{2}(u)\in L^{2}(\Gamma ),$ by %
\eqref{L2-cond-1}, we also have that $\Delta _{\Gamma }u\in L^{2}(\Gamma )$.
Elliptic regularity also implies that $tr\left( u\right) \in H^{2}(\Gamma )$%
. Hence, $U\in D\left( A_{1}\right) \cap D\left( A_{2}\right) ,$ if $q>0$.
It is easy to verify that, for every $U\in D\left( A_{3}\right) =D\left(
A_{1}\right) \cap D\left( A_{2}\right) $, $A_{3}U=A_{1}U+A_{2}U$. The
statement (\ref{conc}) is a straightforward consequence of \eqref{resl-BH-2}%
. The proof is finished.
\end{proof}

The following corollary is a consequence of Proposition \ref{theo-equ}.

\begin{corollary}
Let $\alpha _{1},$ $\alpha _{2}:\mathbb{R}\rightarrow \mathbb{R}$ be
continuous, monotone nondecreasing functions satisfying the growth
conditions \textbf{(GC1)}-\textbf{(GC2)}. Then \eqref{L2-cond}-%
\eqref{L2-cond-1} are fulfilled and therefore, (\ref{conc}) holds.
\end{corollary}

\begin{proof}
To prove this result, we need the following properties of Sobolev spaces.
Since the domain $\Omega $ has smooth boundary $\Gamma $, one has the
following:

\begin{enumerate}
\item If $N=1$, $H^1(\Omega)\hookrightarrow C(\bar\Omega)$.

\item If $N=2$, $H^{1}(\Omega )\hookrightarrow L^{p}(\Omega ),$ for every $%
p\in \lbrack 1,\infty )$ and $H^{1}(\Gamma )\hookrightarrow C(\Gamma )$.

\item If $N\ge 3$, $H^1(\Omega)\hookrightarrow L^{\frac{2N}{N-2}}(\Omega)$.

\item If $N=3$, $H^{1}(\Gamma )\hookrightarrow L^{q}(\Gamma ),$ for every $%
q\in \lbrack 1,\infty )$.

\item If $N\ge 4$, $H^1(\Gamma)\hookrightarrow L^{\frac{2(N-1)}{N-3}%
}(\Gamma) $.
\end{enumerate}

Now, let $\widetilde{\Omega }$ denote either $\Omega $ or $\Gamma $ and
suppose that $q\geq 0$. Then the regularity properties of $u\in $ $%
H^{1}\left( \Omega \right) ,$ if $q=0,$ $u_{\mid \Gamma }\in H^{1}\left(
\Gamma \right) ,$ if $q>0$ given in the five points above, and $\left\vert
\alpha \left( s\right) \right\vert \leq C\left( 1+\left\vert s\right\vert
^{r}\right) $ imply that $\alpha \left( u\right) \in L^{2}(\widetilde{\Omega
}),$ provided that \textbf{(GC1)}-\textbf{(GC2)} are satisfied. In
particular, it is easy to check that $\alpha _{i}\left( u\right) \in L^{2}(%
\widetilde{\Omega }),$ for $i=1,2$. This completes the proof.
\end{proof}

\noindent \textbf{Acknowledgement.} We are most grateful to Haim Brezis for
his interest in this work and for his generous and helpful comments. We also
thank Jean Mawhin for informing us about \cite{Ma}. 

\end{document}